\documentclass[reqno,11pt]{article}
\usepackage{a4wide}
\usepackage{color,enumerate,mathrsfs, amsthm}
\usepackage[normalem]{ulem}
\usepackage{amsmath,amssymb,epsfig,bbm}
\numberwithin{equation}{section}

\usepackage[pdfborder={0 0 0}]{hyperref}  

\usepackage[latin1]{inputenc}
\usepackage[english]{babel}

\renewcommand{\H}{\mathbb{H}}

\newcommand{\N}{\mathbb{N}}

\newcommand{\R}{\mathbb{R}}
\newcommand{\Z}{\mathbb{Z}}

\newcommand{\mm}{{\mbox{\boldmath$m$}}}

\newcommand{\ppi}{{\mbox{\boldmath$\pi$}}}

\newcommand{\sfd}{{\sf d}}

\newcommand{\Kliminf}{K\kern-3pt-\kern-2pt\mathop{\rm lim\,inf}\limits}  
\newcommand{\supp}{\mathop{\rm supp}\nolimits}   
\newcommand{\Lip}{\mathop{\rm Lip}\nolimits}          
\renewcommand{\d}{{\mathrm d}}

\newcommand{\restr}[1]{\lower3pt\hbox{$|_{#1}$}}

\newcommand{\eps}{\varepsilon}  
\newcommand{\nchi}{{\raise.3ex\hbox{$\chi$}}}

\newcommand{\limi}{\varliminf}
\newcommand{\lims}{\varlimsup}

\newcommand{\mean}[1]{\,-\hskip-1.08em\int_{#1}} 

\newcommand{\fr}{\penalty-20\null\hfill$\blacksquare$}                      

\newcommand{\prob}[1]{\mathscr P(#1)}                   
\newcommand{\e}{{\rm{e}}}                           

\renewcommand{\mm}{\mathfrak m}                                

\renewenvironment{proof}{\removelastskip\par\medskip   
\noindent{\em Proof} \rm}{\penalty-20\null\hfill$\square$\par\medbreak}

\newenvironment{proofb}{\removelastskip\par\medskip   
\noindent{\em{\bf Proof}} \rm}{\penalty-20\null\hfill$\square$\par\medbreak}

\newtheorem{theorem}{Theorem}[section]

\newtheorem{corollary}[theorem]{Corollary}
\newtheorem{lemma}[theorem]{Lemma}
\newtheorem{proposition}[theorem]{Proposition}

\newtheorem{definition}[theorem]{Definition}

\newtheorem{remark}[theorem]{Remark}

\newcommand{\X}{{\rm X}}
\newcommand{\Y}{{\rm Y}}
\renewcommand{\Z}{{\rm Z}}

\newcommand{\MCP}{{\sf MCP}}
\newcommand{\CD}{{\sf CD}}
\newcommand{\RCD}{{\sf RCD}}
\newcommand{\ncRCD}{{\sf ncRCD}}
\newcommand{\wncRCD}{{\sf wncRCD}}
\newcommand{\Eu}{{\sf E}}

\newcommand{\dist}{{\rm dist}}

\newcommand{\HS}{{\lower.3ex\hbox{\scriptsize{\sf HS}}}}
\renewcommand{\H}[1]{{\rm Hess}(#1)}

\newcommand{\D}{{\rm D}}

\newcommand{\HH}{{\mathcal H}}
\newcommand{\LL}{{\mathcal L}}

\title{Non-collapsed spaces with Ricci curvature bounded from below}
\begin{document}

\author{ Guido De Philippis\ \thanks{SISSA, gdephili@sissa.it}  \and
   Nicola Gigli
   \thanks{SISSA, ngigli@sissa.it}}

\maketitle

\begin{abstract}
We propose a definition of {\em non-collapsed} space with Ricci curvature bounded from below and we prove the versions of Colding's volume convergence theorem and of Cheeger-Colding dimension gap estimate for ${\sf RCD}$ spaces. 

In particular this establishes the stability of non-collapsed spaces under non-collapsed Gromov-Hausdorff convergence.

\end{abstract}


\tableofcontents

\section{Introduction}
Lott-Villani in \cite{Lott-Villani09} and Sturm in \cite{Sturm06I,Sturm06II} introduced a synthetic notion of lower Ricci curvature bounds for metric measure spaces: their approach is based on suitable convexity properties for entropy-like functionals over the space of probability measures equipped with the quadratic Kantorovich distance $W_2$. The classes of spaces that they introduced are called $\CD(K,N)$, standing for lower Curvature bound by $K\in\R$ and upper Dimension bound by $N\in[1,\infty]$ (in  \cite{Lott-Villani09}  only the cases $K=0$  and $N=\infty$ have been considered). 

Since then the study of these classes of spaces has been a very flourishing research area, see for instance the surveys \cite{Villani2016,Villani2017} and references therein. Among the various fine tunings of Lott-Sturm-Villani's proposal, we mention the definition of $\CD^*(K,N)$ spaces proposed by Bacher-Sturm in \cite{BacherSturm10}: under minor technical conditions, this is locally equivalent to the $\CD(K,N)$, has better local-to-global properties but a priori leads to slightly suboptimal constants in various geometric and functional inequalities (but see \eqref{eq:CM} below).

Since the very beginning, one of the main research lines has been, and still is, that of understanding the geometric properties of such spaces. Here fundamental ideas come from the theory of Ricci-limit spaces developed in the nineties by Cheeger and Colding \cite{Cheeger-Colding96,Cheeger-Colding97I,Cheeger-Colding97II,Cheeger-Colding97III,Colding97}: one would like at least to replicate all their results in the synthetic framework, and then hopefully to obtain, thanks to the new point of view, new insights about both  smooth and non-smooth objects having Ricci curvature bounded from below. In this direction it has been soon realized that the classes of $\CD/\CD^*(K,N)$ spaces are not really suitable for the development of this program: the problem is that Finlser structures are included (see the last theorem in \cite{Villani09}) and for these  Cheeger-Colding's results are not valid. For instance, the Cheeger-Colding-Gromoll splitting theorem fails in finite dimensional Banach spaces. 

Motivated by this problem  the second author proposed in \cite{Gigli12} to reinforce the Lott-Sturm-Villani condition with the functional-analytic notion of infinitesimal Hilbertianity:
\begin{equation}
\label{eq:IH}
\text{$(\X,\sfd,\mm)$ is infinitesimal Hilbertian provided $W^{1,2}(\X,\sfd,\mm)$ is an Hilbert space.}
\end{equation}
This definition is the result of a research program devoted to the understanding of the heat flow \cite{Gigli10,Gigli-Kuwada-Ohta10,AmbrosioGigliSavare11} on $\CD(K,N)$ spaces, and in particular of the introduction of the class of $\RCD(K,\infty)$ spaces  - ${\sf R}$ standing for Riemannian - in a collaboration with Ambrosio and Savar\'e \cite{AmbrosioGigliSavare11-2}. 

In \eqref{eq:IH}, $W^{1,2}(\X,\sfd,\mm)$ is the Sobolev space of real valued functions on $\X$ as introduced by Cheeger in \cite{Cheeger00} (see also the alternative, but equivalent, descriptions provided in \cite{Shanmugalingam00} and \cite{AmbrosioGigliSavare11}). { It is a  priori non-trivial, but nevertheless true, that }infinitesimal Hilbertianity is stable under mGH-convergence  when coupled with a $\CD(K,N)$ condition. Moreover, as proved by the second author in  \cite{Gigli13} (see also \cite{Gigli13over}), the splitting theorem holds in the class of infinitesimally Hilbertian  $\CD(0,N)$ spaces.

In a different direction, in another collaboration \cite{AmbrosioGigliSavare12} of the second author with Ambrosio and Savar\'e it has been introduced the class of ${\sf BE}(K,N)$ spaces: these are spaces in which, in a suitable sense, the Bochner inequality with parameters $K,N$ holds (${\sf BE}$ stands for  Bakry-\'Emery). The key points  of \cite{AmbrosioGigliSavare12} are the proof that the class of ${\sf BE}(K,N)$ spaces is stable under mGH-convergence (and thus provides another reasonable synthetic notion of spaces having a curvature-dimension bound) and that for $N=\infty$ it  coincides with that of $\RCD(K,\infty)$ spaces. 

This circle of ideas has been closed in \cite{Erbar-Kuwada-Sturm13} (and later in \cite{AmbrosioMondinoSavare13}) where it has been proved that 
\[
{\sf BE}(K,N)=\CD^*(K,N)+\text{infinitesimal Hilbertianity}.
\]
More recently, Cavalletti-Milman in \cite{CavMil16} proved in high generality, and in particular without relying on infinitesimal Hilbertianity, that it holds
\begin{equation}
\label{eq:CM}
\CD(K,N)=\CD^*(K,N)\qquad\Big(
\begin{array}{ll}
&\text{under some kind of non-branching assumption}\\
&\text{which always holds in $\RCD(K,\infty)$ spaces}
\end{array}\quad\Big)
\end{equation}
The results in \cite{CavMil16} are stated for  spaces with finite reference measure but the kind of arguments used seems to indicate that the same also holds without this restriction. For  this reason in this manuscript we shall work with $\RCD(K,N):=\CD(K,N)+{\rm Inf.Hilb.}$ spaces, rather than with  $\RCD^*(K,N):=\CD^*(K,N)+{\rm Inf.Hilb.}$ ones which have been recently more popular. In any case, all our arguments are local in nature and since the local versions of $\CD(K,N)$ and $\CD^*(K,N)$ are known to be equivalent from the very first paper \cite{BacherSturm10} where $\CD^*$ has been introduced, our results are independent from \cite{CavMil16}.

\bigskip

We now turn to the description of the content of this manuscript.

Thanks to the celebrated volume convergence result by Colding \cite{Colding97}, and to its generalization to Ricci-limit spaces by Cheeger-Colding \cite{Cheeger-Colding97I}, we know that for a pointed-Gromov-Hausdorff-converging sequence of pointed Riemannian manifolds $(M_n,p_n)$ with the same dimension  and Ricci curvature uniformly bounded from below, the volume of the unit ball around $p_n$ either stays away from 0 (i.e.\ $\inf_n{\rm Vol_n}(B^{M_n}_1(p_n))>0$) or it converges to 0. Limit spaces are called \emph{non-collapsed} or \emph{collapsed} according to whether they are obtained as limits of sequences of the former or latter kind respectively. 

As it turned out from the analysis done in \cite{Cheeger-Colding96,Cheeger-Colding97I,Cheeger-Colding97II,Cheeger-Colding97III}, non-collapsed spaces are more regular than collapsed ones and it is therefore natural to look for a synthetic counterpart of this class of spaces. To do so  we should look for an \emph{intrinsic} characterization of non-collapsed Ricci-limit spaces, i.e.\ for one which does not rely on the existence of a converging sequence having suitable properties. A work in this direction has also been done by Kitabeppu in \cite{Kita17} (see also Remark \ref{re:Kita}).

Let us observe that the aforementioned volume convergence result grants, as noticed in \cite{Cheeger-Colding97I}, that: a pGH-limit space $(\X,\sfd)$ of a sequence of $N$-dimensional manifolds with Ricci curvature uniformly bounded from below is non-collapsed if and only if
\[
\text{the volume measures weakly converge to the measure $\HH^N$ on $\X$ and $\HH^N(\X)>0$},
\]
where here and in the following $\HH^N$ is the  $N$-dimensional Hausdorff measure.

Since for a $\CD(K,N)$ space $(\X,\sfd,\mm)$ the requirement $\mm(\X)>0$ is part of the definition, the above motivates  the following:
\begin{definition}[Non-collapsed $\RCD$ spaces] Let $K\in\R$ and $N\geq 1$. We say that $(\X,\sfd,\mm)$ is a \emph{non-collapsed} $\RCD(K,N)$ space, $\ncRCD(K,N)$ space in short, provided it is an $\RCD(K,N)$ space and $\mm=\HH^N$.
\end{definition}
From the known structural properties of $\RCD(K,N)$ spaces it is not hard to show that if $(\X,\sfd,\mm)$ is  a  $\ncRCD(K,N)$, then   $N$  must be  an  integer. This follows for instance from the rectifiability results proved in \cite{Mondino-Naber14}, \cite{MK16}, \cite{GP16-2} (see Theorem \ref{thm:rett}). Alternatively, this can be proved by blow-up arguments, see Theorem \ref{thm:wncpropp} and in particular the implication $(iii)\Rightarrow (iv)$.

Imitating the arguments in \cite{Cheeger-Colding97I} we shall prove that $\ncRCD(K,N)$ spaces are stable   under Gromov-Hausdorff (thus a priori not necessarily \emph{measured}-GH) convergence in the sense made precise by the following theorem:
\begin{theorem}[Non-collapsed and collapsed convergence]\label{thm:stabnc}
Let $(\X_n,\sfd_n,\HH^N,x_n)$ be a sequence of pointed $\ncRCD(K,N)$ spaces. Assume that $(\X_n,\sfd_n,x_n)$ converges to $(\X_\infty,\sfd_\infty,x)$ in the pointed-Gromov-Hausdorff topology. Then precisely one of the following happens:
\begin{itemize}
\item[i)] $\lims_{n\to\infty}\HH^N(B_1(x_n))>0$. In this case the $\lims$ is actually a limit and $(\X_n,\sfd_n,\HH^N,x_n)$ converges in the pointed-measured-Gromov-Hausdorff topology to $(\X_\infty,\sfd_\infty,\HH^N,x)$. In particular $(\X,_\infty\sfd_\infty,\HH^N)$ is a $\ncRCD(K,N)$ space.
\item[ii)] $\lim_{n\to\infty}\HH^N(B_1(x_n))=0$. In this case $\dim_{\HH}(\X_\infty)\leq N-1$.
\end{itemize}
\end{theorem}
Here and in what follows $\dim_\HH(\X)$ is the Hausdorff dimension of the metric space $\X$.  Notice that in particular
\begin{quote} 
non-collapsed limits of Riemannian manifolds in the sense of Cheeger-Colding are non-collapsed spaces in our sense, 
\end{quote} 
explaining our choice of terminology.

Theorem \ref{thm:stabnc} is strictly related to  the following two results. The first generalizes the already mentioned volume convergence theorem to the $\RCD$ setting. Notice that there is no non-collapsing assumption.
\begin{theorem}[Continuity of $\HH^N$]\label{thm:contvol}
For  $K\in\R$,  $N\in[1,\infty)$ and $R\geq 0$ let $\mathbb B_{K,N,R}$ be the collection of all (equivalence classes up to isometry of) closed  balls of radius $R$ in $\RCD(K,N)$ spaces equipped with the Gromov-Hausdorff distance.

Then the map ${\mathbb B}_{K,N,R}\ni Z\mapsto \HH^N(Z)$ is real valued and continuous.
\end{theorem}
Such theorem is true  even for open balls, see equation \eqref{eq:hnsfere}. Notice also that Gromov precompactness theorem for $\RCD$ spaces and the stability of the $\RCD$ condition grant that $\mathbb B_{K,N,R}$ is compact w.r.t.\ the Gromov-Hausdorff topology (see also the proof of Theorem \ref{thm:contvol} given at the end of Section \ref{se:contvol}).

The second result, analogous to \cite[Theorem 3.1]{Cheeger-Colding97I}, concerns the Hausdorff dimension of an $\RCD(K,N)$ space;  again there is not an assumption about non-collapsing, but on the other hand $N$ is assumed to be integer.
\begin{theorem}[Dimension gap]\label{thm:dimgap}
Let $K\in\R$, $N\in\N$, $N\geq1$, and $\X$ an $\RCD(K,N)$ space. Then either  $\dim_\HH(\X)=N$ or $\dim_\HH(\X)\leq N-1$.
\end{theorem}
Since $\R$ is $\RCD(0,1+\eps)$, we see that  the assumption $N\in\N$ is necessary in the above. { For non-integer $N$'s this last result easily implies the following:
\begin{corollary}\label{cor:ref}
Let $K,N\in\R$, $N\geq1$, and $\X$ an $\RCD(K,N)$ space. Then $\dim_{\HH}(\X)\leq [N]$, where $[\cdot]$ denotes the integer part.
\end{corollary}
Notice that this is sharp because for every $N\in(1,2)$ the space $([0,\pi],\sfd_\Eu,\sin^{N-1}(t)\,\d t)$ is an $\RCD(N-1,N)$ space whose Hausdorff dimension is 1.
}

As a quite direct consequence of Theorem \ref{thm:stabnc} and its proof we also obtain the following volume (almost) rigidity result:
\begin{theorem}[Volume rigidity]\label{thm:volrig}
For every $\eps>0$ and $N\in\N$, $N\geq 1$ there is $\delta=\delta(\eps,N)$ such that the following holds. Let $(\X,\sfd,\HH^N)$ be a $\ncRCD(-\delta,N)$ space and $\bar x\in\X$ such that
\[
\HH^N(B^\X_1(\bar x))\geq \HH^N(B_1^{\R^N}(0))(1-\delta).
\]
Then
\[
\sfd_{\rm GH}\big(\overline B_{1/2}^{\X}(\bar x),\overline B_{1/2}^{\R^N}(0)\big)\leq\eps.
\]
\end{theorem}

The example  of a unit ball in a cylinder shows that we cannot replace $1/2$ with $1$ in the conclusion, see also the discussion in \cite{DPG16}. { A simple consequence of the Bishop-Gromov inequality combined  with Corollary \ref{cor:densnc}  and  the above Theorem is the following:

\begin{corollary}\label{cor:Bishop}
Let $(\X,\sfd,\HH^N)$ be a $\ncRCD(0,N)$ space, then for all \(x\in \X\) and \(r>0\), 
\begin{equation}\label{e:Bishop}
\HH^N(B_{r}^\X(x))\le \omega_Nr^N.
\end{equation}
Moreover, if there exists \(\bar x\in \X\) and \(\bar r>0\) achieving equality in \eqref{e:Bishop}, then \(B_{\bar r/2}^\X(\bar x)\) is isometric to \(B_{r/2}^{\R^N}(0)\). As a consequence, a point \(x\in \X\) is regular (i.e. all tangent cones are isometric to \(\R^N\)) if and only if 
\[
\lim_{r\to 0} \frac{\HH^N(B_{r}^\X(x))}{\omega_N r^N}=1.
\]
\end{corollary}
}

We now pass to the description of the main properties of non-collapsed spaces. A first result is about the stratification of their singular set: denote by $S_k(\X)$ the set of points $x\in\X$ such that no tangent space splits off a factor $\R^{k+1}$ (see \eqref{eq:essek} for the precise definition). In the same spirit of classical stratification results in geometric measure theory first established in \cite{Almgren00} and axiomatized in \cite{White97} we have the following result, compare with  \cite[Theorem 4.7]{Cheeger-Colding97I}.
\begin{theorem}[Stratification]\label{thm:singular}
Let $K\in\R$, {$N\in\N $}, $N\geq 1$ and let $\X$ be a  $\ncRCD(K,N)$ space.
 
Then \(\dim_\HH(\mathcal S_k(\X))\leq k\)  for every $k\in\N$.
 \end{theorem}
Beside these, all the other properties of $\ncRCD$ spaces that we are able to prove hold in the a priori larger class of \emph{weakly} non-collapsed spaces, which we now introduce.

\bigskip

For $K\in\R$, $N\in[1,\infty)$ and $r\geq 0$ let us consider the volume of the ball of radius $r$ in the reference `space form' defined by
\[
v_{K,N}(r):=\omega_N\int_0^r\big|s_{K,N}(t)\big|^{N-1}\,\d t,
\]
where $\omega_N:=\frac{\pi^{N/2}}{\int_0^\infty t^{N/2}e^{-t} \,\d t}$ coincides for integer \(N\) with the volume of the unit ball in \(\R^N\)  and 
\[
s_{K,N}(r):=
\begin{cases}
\sqrt{\frac{N-1}{K}}\sin(r\sqrt{\frac{K}{N-1}}),\qquad&\text{ if }K>0,\\
r,&\text{ if }K=0,\\
\sqrt{\frac{N-1}{|K|}}\sinh(r\sqrt{\frac{|K|}{N-1}}),&\text{ if }K<0.
\end{cases}
\]
Then the Bishop-Gromov inequality, which is valid in the class of $\MCP(K,N)$ spaces {(see \cite{Ohta07} and \cite{Sturm06II} and recall that an $\RCD(K,N)$ space is also $\MCP(K,N)$)}, states that
\begin{equation}\label{e:BG}
r\mapsto \frac{\mm(B_r(x))}{v_{K,N}(r)}\qquad\text{is decreasing}
\end{equation}
for any $x\in\supp(\mm)$. Therefore the following definition if meaningful:
\begin{definition}[Bishop-Gromov density]\label{def:density} Let $K\in\R$, {$N\in [1,\infty)$} and let $(\X,\sfd,\mm)$ be a $\MCP(K,N)$ space with $\supp(\mm)=\X$. For \(x\in \X\) we define the {\em  Bishop-Gromov density} at \(x\) as 
\begin{equation}\label{e:densitydef}
\vartheta_{N}[\X,\sfd,\mm](x):=\lim_{r \to 0} \frac{\mm(B_r(x))}{v_{K,N}(r)}=\sup_{r >0} \frac{\mm(B_r(x))}{v_{K,N}(r)}.
\end{equation}
\end{definition}
Notice that by the very definition of \(v_{K,N}(r)\) we have
\begin{equation}\label{e:limite}
\lim_{r\to 0} \frac{v_{K,N}(r)}{\omega_N r^N}=1,
\end{equation}
hence
\begin{equation}\label{e:densitysempl}
\vartheta_{N}[\X,\sfd,\mm](x)=\lim_{r \to 0} \frac{\mm(B_r(x))}{\omega_Nr^N},
\end{equation}
whence the choice of omitting the $K$ in the notation of the Bishop-Gromov density. Still,   the definition \eqref{e:densitydef} allows to directly exploit \eqref{e:BG} and this simplifies some proofs.

We note that for an \(\RCD(K,N)\) space the Bishop-Gromov density can be equal to \(\infty\) at almost every  point,  a simple example being  the \(\RCD(0,N)\) space \(([0,\infty), \sfd_{\Eu}, x^{N-1}\,\d \LL^1)\), where here and in the sequel  \(\sfd_{\Eu}\) will denote the euclidean distance. In a sense what is happening in this example is that there is a gap between the `functional analytic' upper bound on the dimension \(N\) of the space and  its `geometric' dimension.  This motivates the following:
\begin{definition}[Weakly non-collapsed $\RCD$ spaces]\label{def:wnc} Let $K\in\R$ and {$N\in [1,\infty)$}. We say that  $(\X,\sfd,\mm)$ is a \emph{weakly non-collapsed $\RCD(K,N)$}, $\wncRCD(K,N)$ in short, space  provided it is $\RCD(K,N)$, it holds $\supp(\mm)=\X$ and
\[
\vartheta_{N}[\X,\sfd,\mm](x)<+\infty\qquad\text{for } \mm-\text{a.e. \(x\)}.
\]
\end{definition}
Notice that by classical results about differentiation of measures (see e.g.\ Lemma \ref{le:AT}), if $\HH^N$ is a Radon measure on $\X$ we know that
\[
\lims_{r\downarrow0}\frac{\HH^N(B_r(x))}{\omega_Nr^N}\leq1\qquad\HH^N-a.e.\ x\in\X
\]
and thus in particular
\[
\text{a non-collapsed $\RCD(K,N)$ space is also weakly non-collapsed},
\]
see Corollary \ref{cor:densnc}. 

Also, from \eqref{e:BG} it follows that $\vartheta_{N}$ is lower-semicontinuous both as a function on the fixed $\RCD(K,N)$ space $(\X,\sfd,\mm)$ and along a pmGH-converging sequence (see Lemma \ref{lm:propt}). This easily implies  the stability of the weakly non-collapsed condition  w.r.t.\ pmGH-convergence, see Theorem \ref{thm:stabweak}.
\begin{remark}{\rm
By analogy with the properties of Ricci-limit spaces obtained in \cite{Cheeger-Colding96,Cheeger-Colding97I,Cheeger-Colding97II,Cheeger-Colding97III,ColdingNaber12} we believe that 
\[
\begin{split}
&\text{if $(\X,\sfd,\mm)$ is  $\RCD(K,N)$ and $\vartheta_{N}[\X,\sfd,\mm]<\infty$ on a set of positive $\mm$-measure,}\\
&\text{then up to multiply $\mm$ by a { positive constant} the space is $\ncRCD(K,N)$}
\end{split}
\]
and in particular that any weakly non-collapsed space is, up to multiply the measure by a positive constant,  non-collapsed. Note in particular that a consequence of the above property would be the constancy of the dimension of \(\RCD(K,N)\) spaces in the case when there is at least a \(N\)-dimensional piece\footnote{During the revision process of this manuscript, Bru\`e and Semola proved in \cite{BS18} that finite dimensional $\RCD$ spaces have constant dimension regardless of such `maximality' condition.}. This fact is proved, in full generality,  for  Ricci limit spaces by Colding and Naber in \cite{ColdingNaber12}. 
}\fr\end{remark}
The geometric significance of the finiteness of $\vartheta_N$ is mostly based on the fact that
\begin{equation}
\label{eq:tancon}
\text{if $\vartheta_N[\X,\sfd,\mm](x)<\infty$ then every tangent space at $x$ is a metric cone}
\end{equation}
which in turn follows directly from the `volume cone to metric cone' property of $\RCD$ spaces obtained by the authors in \cite[Theorem 1.1]{DPG16} (see Proposition \ref{le:metriccone} for the proof of \eqref{eq:tancon}).

With this said, we have the following equivalent characterizations  of weakly non-collapsed spaces:

\begin{theorem}\label{thm:wncpropp}
 Let $K\in\R$, {$N\in [1,\infty)$} and let $(\X,\sfd,\mm)$ be a $\RCD(K,N)$ space with \(\supp \mm=\X\). Then the following are equivalent:
 \begin{itemize}
 \item[(i)] \(\X\) is a $\wncRCD(K,N)$ space.
\item[(ii)] \(\mm\ll  \HH^N\). 
 \item[(iii)] There exists a function \(\vartheta_1\in L^1_{loc}(\HH^N)\) such that \(\mm=\vartheta_1 \HH^N\). 
 \item[(iv)]  \(N\) is integer and for \(\mm\)-a.e. \(x\in \X\) there exists a constant \(\vartheta_2(x)\) such that
   \[
    (\X,\sfd/{r},\mm/r^N,x)\overset{\textrm{ pmGH}}{\longrightarrow} (\R^N, \sfd_{\Eu},\vartheta_2 (x)\LL^N,0) \qquad\text{ as }r\downarrow0.
   \]
\item[(v)] \(N\) is integer and for \(\mm\)-a.e. \(x\in \X\) it holds
   \[
    (\X,\sfd/{r},\mm/c_r,x)\overset{\textrm{ pmGH}}{\longrightarrow} (\R^N, \sfd_{\Eu},\LL^N/c(N),0) \qquad\text{ as }r\downarrow0,
   \]
where
\[
c_r:=\int_{B^{\X}_r(x)}\Big(1-\frac{\sfd(y,x)}{r}\Big)\d\mm(y)\qquad c(N):=\int_{B^{\R^N}_1(0)}\big(1-|y|\big)\d\LL^N(y).
\]
\item[(vi)] $N$ is integer and for \(\mm\)-a.e. \(x\in \X\) it holds
   \[
    (\X,\sfd/{r},x)\overset{\textrm{ pGH}}{\longrightarrow} (\R^N, \sfd_{\Eu},0)\qquad\text{ as }r\downarrow0 .
   \]
 \item[(vii)] The  tangent module \(L^2(T\X)\)  has constant dimension equal to  \(N\).
 \end{itemize}
Moreover  in the above statements
 \begin{equation}
\label{eq:stessedens}
 \vartheta_1(x)=\vartheta_2(x)=\vartheta_N[\X,\sfd,\mm](x)<+\infty\qquad\textrm{for \(\mm\)-a.e. \(x\)}.
\end{equation}
Finally, if any of these holds then {(referring to \cite{Gigli14} for the necessary definitions) it holds}
\begin{equation}
\label{eq:hesslap}
H^{2,2}(\X){=} D(\Delta)\qquad\text{ and }\qquad {\rm tr}\,\H f=\Delta f\quad\forall f\in H^{2,2}(\X).
\end{equation}
 \end{theorem}
\begin{remark}{\rm We believe that if $(\X,\sfd,\mm)$ is an $\RCD(K,N)$ space for which \eqref{eq:hesslap} holds, then there exists $n\in\N$, $n\in[1,N]$, such that $(\X,\sfd,\mm)$ is a weakly non-collapsed $\RCD(K,n)$ space. Notice that according to Han's results in \cite{Han14}, this would be true if one knew that the tangent module has constant dimension, in which case one should pick $n$ to be such dimension\footnote{As already mentioned, Bru\`e and Semola recently proved in \cite{BS18} that indeed finite dimensional $\RCD$ spaces have constant dimension. As a consequence of their result, the conjecture in this remark holds.
}. 
}\fr\end{remark}
\begin{remark}\label{re:Kita}{\rm The definition proposed by Kitabeppu in \cite{Kita17} in our formalism reads as:  $\vartheta_N[\X](x)<\infty$ for \emph{every} $x\in\X$ (in particular such  spaces are weakly non-collapsed in our sense). Then in \cite{Kita17}  it has been proved that such spaces have many of the  properties stated in Theorem \ref{thm:wncpropp}, see \cite[Theorem 1.4]{Kita17}, and it has also been noticed that \eqref{eq:tancon} holds. Our proofs of these facts are essentially the same as those in \cite{Kita17}.
}\fr\end{remark}

We conclude mentioning that   the characterization of non-collapsed spaces via blow-ups allows to deduce that `products' and `factorizations' of (weakly) non-collapsed spaces are still (weakly) non-collapsed, see Proposition \ref{prop:prod} and compare it with the non-trivial behaviour - even on $\R^n$ - of products of Hausdorff measures, see e.g.\ \cite[2.10.29]{Federer69}.

\bigskip

{\bf Acknowledgements} The authors wish to thank Jeff Cheeger for a series of inspiring conversations at the early stage of development of this work. { They also would like to thank the anonymous referees for the careful reading of the manuscript and for their comments which helped us to improve the presentation}

\noindent G.D.P.\ is supported by the MIUR SIR-grant `Geometric Variational Problems' (RBSI14RVEZ). 

\noindent N.G.\ is supported by the MIUR SIR-grant `Nonsmooth Differential Geometry' (RBSI147UG4).

\section{Weakly non-collapsed spaces}
\subsection{Stability}\label{se:stab}
In this section we  prove the stability of the class of $\wncRCD(K,N)$ spaces w.r.t.\ pointed-measured-Gromov-Hausdorff convergence.

 In all the upcoming discussion, a \emph{metric space} is always a complete and separable space (sometimes we will consider convergence of open balls in such spaces, but this creates no problems in the definition of Gromov-Hausdorff convergence) and a \emph{metric measure space} is a metric space equipped with a non-negative and {\em non-zero} Radon measure which is finite on bounded sets. { Moreover, by $C(\alpha,\beta,\gamma,\ldots)$ we will always intend a constant whose value depends on the parameters $\alpha,\beta,\gamma,\ldots$ and nothing else.}

Let us begin recalling some basic definitions that will be used throughout the text. The Hausdorff (semi-)distance between two subsets $A,B$ of a metric space $\Y$ is given by
\[
\sfd_{\rm H}(A,B):=\inf\{\eps\geq 0\ :\ B\subset A^\eps\text{ and } A\subset B^\eps\}
\]
where $A^\eps$ denotes the $\eps$-neighbourhood of $A$, i.e. the  set of points at distance $<\eps$ from $A$.

With this said, we now  recall the definitions of the various kind of Gromov-Hausdorff convergences that we shall use. Notice that for the case of pointed and pointed-measured convergences our definitions are not really the correct ones in the general case, but given that we will always deal with geodesic metrics and uniformly locally doubling measures, our (simplified) approach is equivalent to the correct definitions, see for instance the discussions in \cite[Chapter 3]{Gromov07}, \cite[Section 8.1]{BBI01}, \cite[Section 3.5]{GMS15}.

\begin{definition}[Gromov-Hausdorff convergences]\label{def:GH}
Let $(\X_n,\sfd_n)$, $n\in\N\cup\{\infty\}$ be  metric spaces. We say that $(\X_n,\sfd_n)$ converges to $(\X_\infty,\sfd_\infty)$ in the Gromov-Hausdorff (GH in short) sense provided there exist a metric spaces $(\Y,\sfd_\Y)$ and isometric embeddings $\iota_n:\X_n\to\Y$, $n\in\N\cup\{\infty\}$, such that 
\[
\sfd_{\rm H}\big(\iota_n(\X_n),\iota_\infty(\X_\infty)\big)\to 0\qquad\text{ as }n\to\infty.
\]
If the spaces are pointed, i.e.\ selected points $x_n\in\X_n$ are given, we say that $(\X_n,\sfd_n,x_n)$ converges to $(\X_\infty,\sfd_\infty,x_\infty)$ in the pointed-Gromov-Hausdorff (pGH in short) sense provided there exist a metric spaces $(\Y,\sfd_\Y)$ and isometric embeddings $\iota_n:\X_n\to\Y$, $n\in\N\cup\{\infty\}$, such that:
\begin{itemize}
\item[i)] $\iota_n(x_n)\to\iota_\infty(x_\infty)$ in $\Y$,
\item[ii)] for every $R>0$ we have
\[
\sfd_{\rm H}\big(\iota_n(B_R(x_n)),\iota_\infty(B_R(x_\infty))\big)\to 0\qquad\text{ as }n\to\infty.
\]
\end{itemize}
If moreover  the spaces $\X_n$ are  endowed with Radon measures $\mm_n$ finite on bounded sets, we say that $(\X_n,\sfd_n,\mm_n,x_n)$ converges to $(\X_\infty,\sfd_\infty,\mm_\infty,x_\infty)$ in the pointed-measured-Gromov-Hausdorff (pmGH in short)  sense provided there are $\Y$ and $(\iota_n)$ satisfying $(i),(ii)$ above and moreover it holds:
\begin{itemize}
\item[iii)] $((\iota_n)_*\mm_n)$ weakly converges to $(\iota_\infty)_*\mm_\infty$, i.e.\ for every $\varphi\in C_b(\Y)$ with bounded support we have
\[
\int\varphi\,\d(\iota_n)_*\mm_n\to\int\varphi\,\d(\iota_\infty)_*\mm_\infty\qquad\text{ as }n\to\infty.
\]
\end{itemize}
In any of these cases, the collection of the space $\Y$ and isometric embeddings $(\iota_n)$ is called \emph{realization} of the convergence and in any of these cases, given $z_n\in\X_n$, $n\in\N\cup\{\infty\}$, we say that   $(z_n)$ converges to $z_\infty$, and write $z_n\stackrel{GH}\to z_\infty$ provided there exists a realization such that
\[
\lim_{n\to\infty}\sfd_\Y\big(\iota_n(z_n), \iota_\infty(z_\infty)\big)=0.
\]
\end{definition}
Notice that in presence of non-trivial automorphism of the limits space $\X_\infty$ it might be that the same sequence $(z_n)$ converges to two different points $z_\infty,z_\infty'\in\X_\infty$. This creates no issues in the foregoing discussion.

We shall frequently use, without further reference, the fact that the class of $\RCD(K,N)$ spaces is closed w.r.t.\ pmGH-convergence  (see \cite{Lott-Villani09}, \cite{Sturm06I}, \cite{Sturm06II}, \cite{AmbrosioGigliSavare11-2}, \cite{Gigli12}, \cite{GMS15}).

Since the Bishop-Gromov inequality \eqref{e:BG} implies that the measure is locally doubling, we can use Gromov's compactness theorem  (see \cite[Section 5.A]{Gromov07}) to deduce  that
\begin{equation}
\label{eq:Gcomp}
\begin{split}
&\text{if  $(\X_n,\sfd_n,\mm_n,x_n)$, $n\in\N$, are $\RCD(K_n,N)$ spaces with  $N\in[1,\infty)$, $\supp(\mm_n)=\X_n$,}\\
&\text{$\mm_n(B_1(x_n))\in[v,v^{-1}]$ for every $n\in\N$ and some $v\in(0,1)$ and $K_n\to K\in\R$,}\\
&\text{then there is a subsequence pmGH-converging to some $\RCD(K,N)$ space}\\
&\text{$(\X,\sfd,\mm,x)$ with $\supp(\mm)=\X$ and a realization  with $\Y$  proper.}\\
\end{split}
\end{equation} 
{ Recall that a metric space is \emph{proper} provided closed bounded sets are compact. }
Notice that a direct consequence of the definitions is that
\begin{equation}
\label{eq:exseq}
\forall y_\infty\in \X_\infty\quad\text{ there exists }y_n\in\X_n,\ n\in\N\text{ such that }y_n\stackrel{GH}\to y_\infty.
\end{equation}

We also recall that on $\MCP(K,N)$ spaces {(see \cite{Ohta07} and \cite{Sturm06II})} $(\X,\sfd,\mm)$ with $\supp(\mm)=\X$ and $N<\infty$, from the spherical version of Bishop-Gromov inequality  - see \cite[Inequality (2.4)]{Sturm06II} -  it holds
\begin{equation}
\label{eq:boundz}
\mm(B_r(x))=\mm(\bar B_r(x))\qquad\forall x\in\X,\ r>0
\end{equation} 
and in turn this easily implies that if  $(\X_n,\sfd_n,\mm_n,x_n)\stackrel{pmGH}\to(\X_\infty,\sfd_\infty,\mm_\infty,x_\infty)$ and all such spaces are $\MCP(K,N)$ with measures of full support, then
\begin{equation}
\label{eq:convballs}
y_n\stackrel{GH}\to y_\infty\qquad\Rightarrow\qquad\mm_n(B_r(y_n))\to \mm_\infty(B_r(y_\infty))\quad\forall r>0,
\end{equation}
{
which easily follows by the weak convergence of the measures and by the fact that \(\mm (\partial B_{r}(y_\infty))=0\) under our assumptions { (from \eqref{eq:boundz})}.
}
Let us now collect some basic simple properties of the Bishop-Gromov density:
\begin{lemma}[Basic properties of the Bishop-Gromov density]\label{lm:propt} Let  \(K\in \R\), \( N\in [1,\infty)\). Then: 
\begin{itemize}
\item[(i)] Let  \((\X_j,\sfd_j,\mm_j,\bar x_j)\) be a sequence of pointed \(\MCP(K,N)\) spaces pmGH-converging to a limit $\MCP(K,N)$ space \((\X_\infty,\sfd_\infty,\mm_\infty,\bar x_\infty)\). Then 
\[
x_j\stackrel{GH}\to x_\infty\qquad\Rightarrow\qquad\liminf_{j\to \infty} \vartheta_N[\X_j,\sfd_j,\mm_j](x_j) \ge \vartheta_N[\X_\infty,\sfd_\infty,\mm_\infty](x_\infty).
\]
In particular, on a given $\MCP(K,N)$ space $(\X,\sfd,\mm)$, the function $\vartheta_N:\X\to[0,\infty]$  is  lower-semicontinuous (and thus Borel measurable).
\item[(ii)] Let $(\X,\sfd,\mm)$ be  a $\MCP(K,N)$ space. Then \(\mm\)-a.e. point \(x\in \{\vartheta_N<\infty\}\) is an approximate continuity point for \(\vartheta_N\), i.e.
\begin{equation}\label{e:contt}
\lim_{r\to 0} \frac{\mm\Big(\big\{y\in B_r(x): |\vartheta_N(y)-\vartheta_N(x)|>\varepsilon\big\}\Big)}{\mm(B_r(x))}=0
\end{equation}
for every \(\eps>0\).
\item[(iii)]  Let $(\X,\sfd,\mm)$ be  a $\MCP(K,N)$ space and, for \(r>0\),  put \((\X_r,\sfd_r,\mm_r)=(\X,\sfd/r,\mm/r^N)\). Then for every \(x\in \X\) we have \(\vartheta_N[\X_r,\sfd_r,\mm_r](x)=\vartheta_N[\X,\sfd,\mm](x)\).
\end{itemize}
\end{lemma}
\begin{proof} Point $(iii)$ trivially follows from \eqref{e:densitysempl}  and point $(i)$ is a direct consequence of the definitions, of \eqref{eq:convballs} and of the monotonicity granted by the   Bishop Gromov inequality \eqref{e:BG}. 

For point $(ii)$ note that the Bishop-Gromov inequality \eqref{e:BG} grants that \(\mm\) is locally doubling, hence  the Lebesgue differentation Theorem applies to every function \(f\in L^1_{\rm loc}(\X)\):
\begin{equation}\label{e:ld}
\lim_{r\to 0} \frac{1}{\mm(B_r(x))}\int_{B_{r}(x)}|f(y)-f(x)|\d \mm(y)=0\qquad\text{\(\mm\)-a.e. \(x\)}.
\end{equation}
By applying~\eqref{e:ld} to, for instance, \(f(x):=\arctan \vartheta_N(x)\) one easily gets (ii).
\end{proof}
A stability result for the class of $\wncRCD$ spaces now easily follows, see Remark \ref{re:comm} below for some comments on the statement:
\begin{theorem}[Stability of weakly non-collapsed spaces]\label{thm:stabweak}
Let $K\in\R$, $N\in\R$ and let \linebreak $(\X_n,\sfd_n,\mm_n,x_n)$ be a sequence of  $\wncRCD(K,N)$ spaces pmGH-converging to some limit space $(\X,\sfd,\mm,x)$. Assume that for every $R>0$ there is an increasing function $f_R:[0,+\infty]\to[0,+\infty]$ with $f_R(+\infty)=+\infty$ such that
\begin{equation}
\label{eq:lims}
\lims_{n\to\infty}\int_{B^{\X_n}_R(x_n)}f_R\circ\vartheta_N[\X_n,\sfd_n,\mm_n]\,\d\mm_n<\infty.
\end{equation}
Then $(\X,\sfd,\mm)$ is a $\wncRCD(K,N)$ space and for every $R>0$ it holds
\begin{equation}
\label{eq:fatou}
\int_{B^\X_R(x)}f_R\circ\vartheta_N[\X,\sfd,\mm]\,\d\mm\leq\limi_{n\to\infty} \int_{B^{\X_n}_R(x_n)}f_R\circ\vartheta_N[\X_n,\sfd_n,\mm_n]\,\d\mm_n.
\end{equation}
\end{theorem}
\begin{proof} From the stability of the $\RCD$ condition we know that $(\X,\sfd,\mm)$ is $\RCD(K,N)$.
Let $(\Y,\sfd_\Y,(\iota_n))$ be a realization of the pmGH-convergence and define $\tilde\vartheta_n:\Y\to[0,\infty]$ as
\[
\tilde\vartheta_n(y):=\left\{\begin{array}{ll}
\vartheta_N[\X_n,\sfd_n,\mm_n](\iota_n^{-1}(y)),&\qquad\text{ if }y\in\iota(\X_n),\\
+\infty,&\qquad\text{ otherwise}.
\end{array}\right.
\]
and similarly $\tilde\vartheta$. Then from point $(i)$ of Lemma \ref{lm:propt} and the monotonicity of $f_R$ we deduce that
\[
y_n\to y\qquad\Rightarrow\qquad {\nchi_{B_R^\Y(\iota(x))}(y)} f_R(\tilde\vartheta(y))\leq\limi_{n\to\infty}{\nchi_{B_R^\Y(\iota_n(x_n))}(y_n)} f_R(\tilde\vartheta_n(y_n))
\]
 { having also used the  fact that if $y\in B_R^\Y(\iota(x))$ then eventually  $y_n\in B_R^\Y(\iota_n(x_n))$.} By the simple Lemma \ref{le:fatou} below  this inequality and the weak convergence of $(\iota_n)_*\mm_n$ to $\iota_*\mm$ give
\[
\int\nchi_{B_R^\Y(\iota(x))}f_R\circ\tilde \vartheta\,\d \iota_*\mm\leq\limi_{n\to\infty} \int\nchi_{B^{\Y}_R(\iota_n(x_n))}f_R\circ\tilde \vartheta_n\,\d(\iota_n)_*\mm_n,
\]
which is \eqref{eq:fatou}. In particular, taking into account \eqref{eq:lims} we deduce that
\[
\int_{B^\X_R(x)}f_R\circ\vartheta_N[\X,\sfd,\mm]\,\d\mm<\infty
\]
and since $f_R(+\infty)=+\infty$, this forces $\vartheta_N[\X,\sfd,\mm]$ to be finite $\mm$-a.e..
\end{proof}
\begin{remark}\label{re:comm}{\rm It is not hard to check that, in  this last theorem, if all the spaces $\X_n$ are Riemannian manifolds of  the same dimension $k$ converging to a smooth Riemannian manifold \(\X\), then necessarily \(k=N\) and  \(\X\) has dimension \(k\). In particular, the convergence is non-collapsed in the sense of Cheeger and Colding, \cite{Cheeger-Colding97I}.

In this respect the following example might be explanatory:  Let $S^1_r$ be the 1-dimensional sphere of radius $r$ and consider the cylinder $M_n:=\R\times S^1_{1/n}$ equipped with its natural product distance $\sfd_n$ and volume measure ${\rm vol}_n$. It is clear that as $n\to\infty$ the metric spaces $(M_n,\sfd_n)$ converge in the pGH-topology to the real line, which trivially has smaller dimension (since in all these manifolds the isometry group acts transitively, the  choice of reference point is irrelevant and thus omitted). 

Let us now consider convergence of the metric \emph{measure} spaces $(M_n,\sfd_n,{\rm vol}_n)$. Notice that  for any $r>1/n$ and $p_n\in M_n$ we have that 
\begin{equation}
\label{eq:pallepiccole}
{\rm vol}_n(B_r(p_n))\sim \frac rn\qquad\text{ as }\qquad n\to\infty
\end{equation}
and thus the measures $\mm_n$ weakly converge, in any realisation of the pGH-convergence, to 0. However, the choice of the null measure is excluded by the definition of metric measure space - see the beginning of Section \ref{se:stab} -, so that $(\R,\sfd_{\Eu},0)$ is not a legitimate metric measure space and the spaces $(M_n,\sfd_n,{\rm vol}_n)$ do not satisfy the assumptions of Theorem \ref{thm:stabweak} above, because they do not converge anywhere in the pmGH-topology. We emphasize that the choice of excluding null reference measures is customary in this research field, see for instance \cite{GMS15} and references therein for a discussion of this topic in relation to convergence of mm-structures.

The typical way to avoid measures disappearing in the limit is to renormalise them via the multiplication by an appropriate constant: this is precisely what Cheeger-Colding do in \cite{Cheeger-Colding97I}, \cite{Cheeger-Colding97II}, \cite{Cheeger-Colding97III} when defining renormalised limit measure. In our case, by \eqref{eq:pallepiccole} we are led to consider the measures $\mm_n:=c_n{\rm vol}_n$ with $c_n\sim n$, so that the spaces $(M_n,\sfd_n,\mm_n)$ converge in the pmGH-sense to $(\R,\sfd_\Eu,c\mathcal L^1)$ for some $c>0$. Thus  we have $\vartheta_2[M_n,\sfd_n,\mm_n]\equiv c_n\to+\infty$ and $\mm_n(B_r(p_n))\to c\mathcal L^1(B_r(0))=2cr>0$ as $n\to\infty$. Hence for  any  function $z\mapsto f(z)$ going to $+\infty$ as $z\to+\infty$ we have
\[
\int_{B_r(p_n)}f\circ\vartheta_2[M_n,\sfd_n,\mm_n]\,\d\mm_n=f(c_n)\,\mm_n(B_r(p_n))\to+\infty,\qquad\text{ as }n\to+\infty,
\]
so that the assumption \eqref{eq:lims} does not hold in this case.

We conclude pointing out that the notion of (weakly) non-collapsed space makes sense only when coupled with the dimension which is being considered, so that it can very well be that a sequence of 2-dimensional non-collapsed spaces converges, with collapsing, to a 1-dimensional non-collapsed space. This is precisely what happens in the example we discussed here.
}\fr\end{remark}

In the proof of  Theorem \ref{thm:stabweak} we used the following known simple variant of the classical Fatou lemma:
\begin{lemma}[A variant of Fatou's lemma]\label{le:fatou}
Let $(\Y,\sfd_\Y)$ be a complete and separable metric space, $\{\mu_n\}_{n\in\N\cup\{\infty\}}$ be Radon measures finite on bounded sets such that
\[
\lim_{n\to\infty}\int \varphi\,\d\mu_n=\int\varphi\,\d\mu_\infty
\]
for every $\varphi\in C_b(\Y)$ with bounded support. Also, let $f_n:\Y\to\R\cup\{+\infty\}$, $n\in\N\cup\{\infty\}$  be such that 
\begin{equation}
\label{eq:glimi}
y_n\to y\qquad\Rightarrow\qquad f_\infty (y)\leq\limi_{n\to\infty}f_n(y_n)
\end{equation}
and $f_n\geq g$ for every $n\in\N\cup\{\infty\}$ for some $g\in C_b(\Y)$ with bounded support.

Then
\[
\int f_\infty\,\d\mu_\infty\leq\limi_{n\to\infty}\int f_n\,\d\mu_n.
\]
\end{lemma}
\begin{proof}
Replacing $f_n$ with $f_n-g$ we can assume that the $f_n$'s are non-negative. Then we follow verbatim the proof in \cite[Lemma 8.2]{AmbDepKir11} which, although presented on $\R^d$, actually holds also in our context.
\end{proof}
\subsection{Tangent spaces}
In this section we study the tangent spaces of weakly non-collapsed spaces, here is the definition that we will adopt (notice the chosen scaling of the measure):
\begin{definition}[(metric) tangent space]
Let $K\in\R$, $N\in[1,\infty)$, $(\X,\sfd,\mm)$ an $\RCD(K,N)$ space with $\supp(\mm)=\X$ and $x\in\X$.

We say that $(\Y,\sfd_\Y,o)$ is a \emph{metric tangent space} of $\X$ at $x$ if there exists a sequence $r_n\downarrow0$ such that
\[
(\X_{r_n},\sfd_{r_n},x):=(\X,\sfd/r_n,x)\quad\stackrel{pGH}\to\quad (\Y,\sfd_\Y,o)\qquad\text{ as }n\uparrow\infty.
\]
Similarly, we say that $(\Y,\sfd_\Y,\mm_\Y,o)$ is a \emph{tangent space} of $\X$ at $x$ if there exists a sequence $r_n\downarrow0$ such that
\[
(\X_{r_n},\sfd_{r_n},\mm_{r_n},x):=(\X,\sfd/r_n,\mm/{r_n^N},x)\quad \stackrel{pmGH}\to\quad(\Y,\sfd_\Y,\mm_\Y,o)\qquad\text{ as }n\uparrow\infty.
\]
\end{definition} 
Notice that the Bishop-Gromov inequality \eqref{e:BG} gives that $\inf_{r\in(0,1)}\mm_r(B_1(x))>0$ for every $x\in\supp(\mm)$ and if $\vartheta_N(x)<\infty$ then by \eqref{e:densitysempl} we also have that $\sup_{r\in(0,1)}\mm_r(B_1(x))<\infty$. Hence recalling \eqref{eq:Gcomp} we see that given an $\RCD(K,N)$ space $(\X,\sfd,\mm)$ and a point $x\in \X$ with $\vartheta_N(x)<\infty$, the family $(\X_r,\sfd_r,\mm_r,x)$, $r\in(0,1)$, is precompact and{, by scaling, }any limit space as $r\downarrow 0$ is $\RCD(0,N)$.

For the definition of cone built over a metric space see for instance \cite[Definition 3.6.16]{BBI01}. We then give the following:
\begin{definition}[Metric (measure) cones]
We say that $(\X,\sfd)$ is a metric cone with vertex $x\in\X$ provided there is a metric space $(\Z,\sfd_\Z)$ and an isometry $\iota$ between $\X$ and the cone over $\Z$ sending $x$ to the vertex.

If $\X$ is also endowed with a Radon measure $\mm$ we say that  it is a metric measure cone provided there are $\Z,\iota$ as before and moreover there are a Radon measure $\mm_\Z$ on $\Z$ and $\alpha\ge 1$ such that 
\[
\d(\iota_*\mm)(r,z)=\d r\otimes r^{\alpha-1}\,\d\mm_\Z(z).
\]
In this case we say that $\X$ is an $\alpha$-metric measure cone.
\end{definition}
A crucial regularity property of weakly non-collapsed spaces is contained in the following statement, which in turn is a   direct  consequence of the `volume cone to metric cone' for $\RCD$ spaces obtained in \cite[Theorem 1.1]{DPG16}:
\begin{proposition}[Tangent spaces are cones]\label{le:metriccone}
Let $K\in\R$, $N\in[1,\infty)$, $(\X,\sfd,\mm)$  an $\RCD(K,N)$ space and $\bar x\in\X$ such that $\vartheta_N(\bar x)<\infty$. Then every tangent space $(\X_\infty,\d_\infty,\mm_\infty,o)$ at $\bar x$ is an $N$-metric measure cone based in $o$ and it holds
\begin{equation}
\label{eq:stessadens}
\vartheta_N[\X,\d,\mm](\bar x)=\vartheta_N[\X_\infty,\d_\infty,\mm_\infty](o)=\frac{\mm_\infty(B_\varrho(o))}{\omega_N\varrho^N}\qquad\forall\varrho>0.
\end{equation}
\end{proposition}
\begin{proof}
Let $r_n\downarrow0$ be such that the rescaled spaces $(\X_{r_n},\sfd/r_n,\mm/r_n^N,\bar x)$  pmGH-converge to the $\RCD(0,N)$ space $(\X_\infty,\d_\infty,\mm_\infty,o)$. We shall apply \cite[Theorem 1.1]{DPG16} to the space $\X_\infty$. From the very definition of pmGH-convergence and recalling \eqref{eq:convballs}, for any $\varrho>0$ we have 
\begin{equation}
\label{eq:percono}
\frac{\mm_\infty(B_\varrho(o))}{\omega_N\varrho^N}=\lim_{n\to\infty}\frac{\mm_\infty(B_{r_n\varrho}(\bar x))}{\omega_N(r_n\varrho)^N}=\vartheta_N[\X,\d,\mm](\bar x).
\end{equation}
Hence $\varrho\mapsto\frac{\mm_\infty(B_\varrho(o))}{\varrho^N}$ is constant and according to \cite[Theorem 1.1]{DPG16} this is enough to deduce that $\X_\infty$ is a $N$-metric measure cone based in $o$. Also, letting $\varrho\downarrow0$ in \eqref{eq:percono} we deduce \eqref{eq:stessadens}.
\end{proof}
The fact that tangent cones of $\wncRCD$ spaces are in fact a.e.\ Euclidean spaces is based on the following simple lemma. Notice that the first part of the statement only assumes the space to be a metric cone, and not a metric measure cone: the rigidity is possible because the splitting theorem for $\RCD(0,N)$ spaces only requires the existence of a straight line on the given space and this is a metric requirement (as opposed to a metric-measure requirement). 
\begin{lemma}\label{lm:favaa}
 Let  {$N\in [1,\infty)$} and let $(\X,\sfd,\mm)$ be an  $\RCD(0,N)$ space which, for every $x\in\X$,  is a metric  cone with vertex in \(x\). Then there exists \(m\in \N\) and \(c_m>0\)  such that \((\X,\sfd,\mm)=(\R^m,\sfd_{\rm  E}, c_m \LL^m)\). 
 
 If we also know a priori that $\X$ is an $N$-metric measure cone with vertex $x$ for every $x\in\X$, then $N\in\N$ and $m=N$ in the above.
\end{lemma}
\begin{proof}
By the very definition of metric cone with vertex \(\bar x\) any point \(x\in \X\setminus\{\bar x\}\) lies in the interior of a half line (i.e. an isometric embedding of \([0,+\infty)\)). Moreover, by assumption, for every \(x\in \X\setminus\{\bar x\}\) and \(r>0\)  the pointed  spaces  \((\X,\sfd/r,x)\) and \((X,\sfd,x)\) are isometric and therefore any metric tangent space at $x$ must coincide with $\X$ itself. Given that $x$ lies in the interior of a length minimising geodesic, the tangent space, and hence $\X$ itself, must contain a line through $x$, see for instance \cite[Proof of Theorem 1.1]{GigliMondinoRajala15} for a similar argument. Thus the splitting theorem for $\RCD$ spaces \cite{Gigli13}, \cite{Gigli13over} grants that  \((\X,\sfd,\mm)\) splits off a line{, i.e.\ it is isomorphic to the product of the Euclidean line $\R$ and a metric measure space \((\X',\sfd',\mm')\). Moreover, such $\X'$ is a point if $N\in[1,2)$ and a $\RCD(0,N-1)$ space if $N\geq 2$}. By iterating this fact finitely many times we obtain the desired conclusion.

The last statement is now obvious.
\end{proof}
We then have the following:
\begin{proposition}\label{prop:tang}
Let $K\in\R$, $N\in[1,\infty)$, $(\X,\sfd,\mm)$ an $\RCD(K,N)$ space with $\supp(\mm)=\X$ and $\bar x\in\X$. Assume that $\bar x$ is a point of approximate continuity of $\vartheta_N[\X]$, i.e.\ $\vartheta_N[\X](\bar x)<\infty$ and \eqref{e:contt}  holds.

Then $N\in\N$ and $(\R^N,\sfd_\Eu,\vartheta_N[\X](\bar x)\mathcal L^N,0)$ is the only tangent space of $\X$ at $\bar x$.
\end{proposition}
\begin{proof}
Let \((\X_\infty, \sfd_\infty,\mm_\infty,o)\) be a tangent space at \(\bar x\), let \(r_n\downarrow 0\) be a sequence that realises it and pick $y\in \X_\infty$. We claim that there exists a sequence $n\mapsto y_n\in\X_{r_n}$ such that \(y_n\overset{\text{GH}}{\rightarrow} y\) and
\begin{equation}
\label{eq:conttheta}
 \vartheta_N[X, \sfd, \mm](y_n)\to \vartheta_N[X, \sfd, \mm](\bar x).
\end{equation}
Indeed, let $n\mapsto \tilde y_n\in \X_{r_n}=\X$ be arbitrary such that  \(\tilde y_n\overset{\text{GH}}{\rightarrow} y\) (recall \eqref{eq:exseq}), notice that $\sfd(\tilde y_{r_n},\bar x)\to 0$ and  that the choice of $\bar x$ and the fact that $\mm$ is doubling grant that for every $r,\eps>0$ the balls $B_{rr_n}(\tilde y_n)\subset \X$ must eventually intersect the set $\{x:|\vartheta(x)-\vartheta(\bar x)|<\eps\}$. Hence with a perturbation and diagonalization argument, starting from $(\tilde y_n)$ we can produce the desired $(y_n)$.

With this said, for any   \(\varrho>0\) we have
\begin{equation}
\label{e:limsupp}
\begin{split}
\frac{\mm_\infty(B_\varrho(y))}{\omega_N\varrho^N}&\stackrel{\eqref{eq:convballs}}=\lim_{n\to \infty}\frac{\mm(B_{\varrho r_n} (y_n))}{\omega_N(\varrho r_n)^N}\stackrel{\eqref{e:limite},\eqref{e:BG}}\le \limi_{n\to\infty}   \vartheta_N[X, \sfd, \mm](y_n)\stackrel{\eqref{eq:conttheta}}=\vartheta_N[X, \sfd, \mm](\bar x).
\end{split}
\end{equation}
On the other hand, putting $R:=\sfd_\infty(y,o)$ and using again  the Bishop-Gromov inequality \eqref{e:BG} (recall that $\X_\infty$ is $\RCD(0,N)$) we have
\begin{equation}
\label{e:liminff}
\begin{split}
\frac{\mm_\infty(B_\varrho(y))}{\omega_N\varrho^N}&\stackrel{\eqref{e:BG}}\ge \lim_{r\to \infty} \frac{\mm_\infty(B_r(y))}{\omega_Nr^N}=\lim_{r\to \infty} \frac{\mm_\infty(B_{r+R}(y))}{\omega_N(r+R)^N}\\
&\ge \lim_{r\to \infty} \frac{\mm_\infty(B_{r}(o))}{\omega_Nr^N} \frac{r^N}{(r+R)^N}\stackrel{\eqref{eq:stessadens}}=\vartheta_N[X, \sfd, \mm](\bar x).
\end{split}
\end{equation}
From  \eqref{e:limsupp} and \eqref{e:liminff} we deduce that 
\begin{equation}
\label{eq:rN}
\text{$\varrho\mapsto \frac{\mm_\infty(B_\varrho(y))}{\omega_N\varrho^N}$ is constantly equal to }\vartheta_N[X, \sfd, \mm](\bar x)
\end{equation} 
and from \cite{DPG16} we can then deduce that $(\X_\infty,\sfd_\infty,\mm_\infty,y)$ is a $N$-metric measure cone. Then  arbitrariness of  $y\in \X_\infty$ and the simple Lemma \ref{lm:favaa} above give the conclusion.
\end{proof}

\subsection{Equivalent characterizations of weakly non-collapsed spaces}
Here we shall prove Theorem \ref{thm:wncpropp} about different equivalent characterizations of weakly non-collapsed spaces.

We shall make use of the  following classical result about differentiation of measures, see e.g.\ \cite[Theorem 2.4.3]{AmbrosioTilli04} for the proof.
\begin{lemma}[Density w.r.t.\ Hausdorff measures]\label{le:AT} Let $(\X,\sfd)$ be a complete and separable metric space, $\mm$ a Radon measure on it and for $\alpha\geq 0$ define the $\alpha$-upper density function as:
\[
\bar\vartheta_\alpha(\mm,x):=\lims_{r\downarrow0}\frac{\mm(B_r(x))}{\omega_\alpha r^\alpha}.
\]
Then for every Borel $B\subset \X$ and $c>0$ it holds
\begin{align}
\label{eq:basso}
\bar\vartheta_\alpha(\mm,x)&\geq c\qquad\forall x\in B\qquad\Rightarrow\qquad \mm(B)\geq c\HH^\alpha(B),\\
\label{eq:aalto}
\bar\vartheta_\alpha(\mm,x)&\leq c\qquad\forall x\in B\qquad\Rightarrow\qquad \mm(B)\leq c2^\alpha\HH^\alpha(B),
\end{align}
\end{lemma}
Let us point out a direct consequence of the above which, being based on the Bishop-Gromov inequality only, is valid on general $\MCP(K,N)$ spaces:
\begin{proposition}\label{prop:basehn}
Let $K\in\R$, $N\in[1,\infty)$ and  $(\X,\sfd,\mm)$  a $\MCP(K,N)$ space with $\supp(\mm)=\X$. Then for every $R>0$ there is $C=C(K,N,R)$ such that for every $x\in\X$   it holds
\begin{equation}
\label{eq:hnm}
\HH^N\restr{B_R(x)}\leq\frac{ C(K,N,R)}{\mm(B_1(x))}\mm\restr{B_R(x)}.
\end{equation}
In particular, $\HH^N$ is a Radon measure on $\X$, is absolutely continuous w.r.t.\ $\mm$ and it holds
\begin{equation}
\label{eq:hnsfere}
\HH^N(B_r(x))=\HH^N(\bar B_r(x))\qquad\forall x\in\X,\ r>0.
\end{equation} 
\end{proposition}
\begin{proof} The Bishop-Gromov inequality \eqref{e:BG} implies that $\vartheta_{N}[\X,\sfd,\mm](y)\geq \frac{\mm(B_{2R}(y))}{v_{K,N}(2R)}\geq  \frac{\mm(B_{1}(x))}{v_{K,N}(2R)}$ for every $y\in B_R(x)$ { and  $R>1$. Also,}  from \eqref{e:densitysempl} we know that $\bar\vartheta_N(\mm,y)=\vartheta_{N}[\X,\sfd,\mm](y)$ for every $y\in\X$. Hence \eqref{eq:hnm}  comes from  \eqref{eq:basso} and then \eqref{eq:hnsfere} follows from \eqref{eq:boundz}.
\end{proof}
Before coming to the proof of Theorem \ref{thm:wncpropp} let us collect in the following statement the known rectifiability properties of $\RCD$ spaces:
\begin{theorem}[Rectifiability of $\RCD$ spaces]\label{thm:rett}
Let $K\in\R$, $N\in[1,\infty)$ and $(\X,\sfd,\mm)$ be an $\RCD(K,N)$ space. Then we can write
\begin{equation}
\label{eq:decX}
\X=\mathcal N\cup\bigcup_{k=1}^M\bigcup_{j\in \N} U^k_j
\end{equation}
for Borel sets $\mathcal N, U^k_j$ where  $\mm(\mathcal N)=0$,  \(M\in \N\), \(M\le N\),    each \(U^k_j\) is  bi-Lipschitz to a subset of \(\R^{k}\),  and for $\mm$-a.e.\ $x\in U^k_j$ the metric tangent space  at $x$ is the Euclidean space $\R^k$.
Moreover for any $j,k$ it holds
\begin{equation}
\label{eq:mac}
\mm \restr{ U^k_j}=\vartheta_j^k\HH^{k} \restr{ U^k_j}
\end{equation}
for some Borel function $\vartheta_j^k:\X\to\R$ which also satisfies
\begin{equation}
\label{eq:densk}
\vartheta_j^k(x)=\lim_{r\downarrow0}\frac{\mm(B_r(x)\cap U_j^k)}{\omega_kr^k}=\lim_{r\downarrow0}\frac{\mm(B_r(x))}{\omega_kr^k}\qquad \HH^k\restr{U_j^k}-a.e.\ x.
\end{equation} 
\end{theorem}
\begin{proof} The existence of the partition \eqref{eq:decX}, of bi-Lipschitz charts and the fact that metric tangent spaces are Euclidean have all been proved in  \cite{Mondino-Naber14}. Property \eqref{eq:mac} has been proved in  \cite{MK16}, \cite{GP16-2}. These informations together grant that $\mm\restr{U_j^k}$ is a $k$-rectifiable measure  according to \cite[Definition 5.3]{AmbKir00}, hence the first equality in\eqref{eq:densk} follows from  \cite[Theorem 5.4]{AmbKir00}. To conclude, notice that if 
\[
\vartheta_j^k(x)<\lims_{r\downarrow0}\frac{\mm(B_r(x))}{\omega_kr^k}
\]
holds in a Borel set $A\subset U_j^k$ of positive $\HH^k$-measure, then we can find $b>a\geq 0$ and a Borel set $A'\subset A$ such that $\HH^N(A')>0$,  $\vartheta_j^k\leq a$ $\HH^k$-a.e.\ on $A'$ and $\lims_{r\downarrow0}\frac{\mm(B_r(x))}{\omega_kr^k}\geq b$ for $x\in A'$. 

This would lead to 
\[
\mm(A')\stackrel{\eqref{eq:basso}}\geq b\HH^k(A')>a\HH^N(A')\stackrel{\eqref{eq:mac}}\geq  \mm(A'),
\]
which is impossible. This proves the second equality in \eqref{eq:densk} and concludes the proof { (see also \cite{AHT17} for similar arguments)}.
\end{proof}

We are now ready to prove Theorem~\ref{thm:wncpropp}.
\begin{proofb}{\bf of Theorem \ref{thm:wncpropp}}\\
\noindent{\bf{(i) $\Rightarrow$ (ii)}} By \eqref{eq:aalto} we know that $\mm\restr{\{\theta<+\infty\}}\ll\HH^N$ and since by hypothesis we have that $\mm(\{\vartheta_N=+\infty\})=0$, the claim follows. 

\noindent{\bf{(ii) $\Rightarrow$ (iii)}} Proposition \ref{prop:basehn} grants that $\HH^N$ is $\sigma$-finite, hence the claim follows by the Radon-Nikodym theorem.

\noindent{\bf{(iii) $\Rightarrow$ (i)}} We consider the decomposition \eqref{eq:decX} and notice that the assumption  \(\mm=\vartheta_1\HH^N\) and \eqref{eq:mac}  forces \(\mm(U^k_j)=0\) for every \(k<N\) and \(j\in \N\) and, since $\mm(X)>0$, \(N\) to be an integer. Hence for every $j$ we have
\begin{equation}
\label{eq:per19}
\vartheta_N[\X](x)\stackrel{\eqref{e:densitysempl}}=\lim_{r\to 0}\frac{\mm(B_r(x))}{\omega_Nr^N}\stackrel{\eqref{eq:densk}}=\vartheta_1(x)<+\infty\qquad\textrm{for \(\mm\restr{U_j^N}\)-a.e. \(x\)}.
\end{equation}
\noindent{\bf{(i) $\Rightarrow$ (iv)}} Consequence of the assumptions, point $(ii)$ of Lemma \ref{lm:propt} and Proposition \ref{prop:tang}, which also grant that 
\begin{equation}
\label{eq:per192}
\vartheta_N=\vartheta_2\qquad \mm\textrm{-a.e.}.
\end{equation}
\noindent{\bf{(iv) $\Rightarrow$ (v)}} This is immediate, since one can easily check that
\[
c_r/r^N=\int_{B_1^{\X_r}}(1-\sfd_r(y,x))\d \mm_r(y)\qquad\to\qquad{ \vartheta_2(x)} \int_{B_1^{\R^N}} (1-|y|)\d\LL^N(y). 
\]
\noindent{\bf{(v) $\Rightarrow$ (vi)}}  Trivial by definitions.

\noindent{\bf{(vi) $\Rightarrow$ (ii)}} By Theorem \ref{thm:rett} we know that for every $k,j$, for $\mm$-a.e.\ $x\in U^k_j$ the metric tangent space at $x$ is $\R^k$. Thus our assumption forces $\mm$ to be concentrated on $\cup_j U^N_j$ and the conclusion follows recalling \eqref{eq:mac}.

\noindent{\bf{(vi) $\Leftrightarrow$ (vii)}} This is an immediate consequence of~\cite[Theorem 5.1]{GP16}.

\noindent{\bf{Proof of \eqref{eq:stessedens}}} Consequence of \eqref{eq:per19} and \eqref{eq:per192}.

\noindent{\bf{(vii) $\Rightarrow$ \eqref{eq:hesslap}}} This follows from~\cite[Proposition 4.1]{Han14}.
\end{proofb}
An easy consequence of the above is:
\begin{corollary}[$\ncRCD\Rightarrow\wncRCD$]\label{cor:densnc}
Let $(\X,\sfd,\HH^N)$ be  a $\ncRCD(K,N)$ space. Then
\begin{equation}
\label{eq:cor}
\vartheta_N(x)\leq 1\qquad\forall x\in\X.
\end{equation}
In particular, {$(\tilde\X,\sfd,\HH^N)$} is $\wncRCD(K,N)${, where $\tilde \X\subset \X$ is the support of $\HH^N$}.
\end{corollary}
\begin{proof}
By point $(iii)$ of Theorem \ref{thm:wncpropp} and \eqref{eq:stessedens} we see that $\vartheta_N\leq 1$ $\HH^N$-a.e.. Then { \eqref{eq:cor}} follows by the lower semicontinuity of $\vartheta_N$ established in point $(i)$ of Lemma \ref{lm:propt}. \end{proof}

\subsection{Tensorization and factorization}
Given two metric measure spaces $(\X_1,\sfd_1,\mm_1)$ and $(\X_2,\sfd_2,\mm_2)$, by their product we mean the product  $\X_1\times\X_2$ equipped with the distance $\sfd_1\otimes\sfd_2$ defined by
\[
(\sfd_1\otimes\sfd_2)^2\big((x_1,x_2),(x_1',x_2')\big):=\sfd_1^2(x_1,x_1')+\sfd_2^2(x_2,x_2')\qquad \forall x_1,x_1'\in\X_1,\ x_2,x_2'\in\X_2
\]
and the product measure $\mm_1\times\mm_2$.

Recall that the product of an $\RCD(K,N_1)$ and an $\RCD(K,N_2)$ space is  $\RCD(K,N_1+N_2)$  (see \cite{Sturm06II}, \cite{AmbrosioGigliSavare11-2}, \cite{AmbrosioGigliSavare12}).

With this said, thanks to characterization of $\wncRCD$ spaces via blow-ups obtained in Theorem \ref{thm:wncpropp} we can easily prove that products and factors of $\wncRCD$ (resp.\ $\ncRCD$) are $\wncRCD$ (resp.\ $\ncRCD$):
\begin{proposition}[Tensorization and factorization of non-collapsed spaces]\label{prop:prod}
Let $(\X_i,\sfd_i,\mm_i)$ be $\RCD(K,N_i)$ spaces, $i=1,2$, with $K\in\R$ and $N_i\in[1,\infty)$ and consider the product space $(\X_1\times\X_2,\sfd_1\otimes\sfd_2,\mm_1\times\mm_2)$.

Then $\X_1\times\X_2$ is $\wncRCD(K,N_1+N_2)$ if and only if $\X_1$ is $\wncRCD(K,N_1)$ and $\X_2$ is $\wncRCD(K,N_2)$.

Similarly, $\X_1\times\X_2$ is $\ncRCD(K,N_1+N_2)$ if and only if for some constant $c>0$ $(\X_1,\sfd_1,c\mm_1)$ is $\ncRCD(K,N_1)$ and $(\X_2,\sfd_2,c^{-1}\mm_2)$ is $\ncRCD(K,N_2)$.
\end{proposition}
\begin{proof} From Theorem \ref{thm:rett} we know that for $\mm_1$-a.e.\ $x_1$ the metric tangent space of $\X_1$ at $x_1$ is $\R^{n_1(x_1)}$ with $n_1(x_1)\leq N_1$. Similarly for $\X_2$. Then from the  very definition of pGH-convergence and Fubini's theorem it is readily checked that $\R^{n_1(x_1)+n_2(x_2)}$ is the metric tangent space of $\X_1\times\X_2$ at $(x_1,x_2)$ for $\mm_1\times\mm_2$-a.e.\ $(x_1,x_2)$. 

Thus the claims about $\wncRCD$ spaces follows by the characterization given in point $(vi)$ of Theorem \ref{thm:wncpropp}.

For the case of $\ncRCD$ spaces we can assume, by what just proved, that $\X_1,\X_2,\X_1\times\X_2$ are all $\wncRCD$ spaces. Then we notice that, much like in the metric case just considered, if $(\R^{N_1},\sfd_\Eu,\vartheta_{N_1}[\X_1](x_1)\mathcal L^{N_1},0)$ (resp. $(\R^{N_2},\sfd_\Eu,\vartheta_{N_2}[\X_2](x_2)\mathcal L^{N_2},0)$) is the tangent space of $\X_1$ (resp. $\X_2$) at $x_1$ (resp. $x_2$), then $(\R^{N_1+N_2},\sfd_\Eu,\vartheta_{N_1}(x_1)\vartheta_{N_2}(x_2)\mathcal L^{N_1+N_2},0)$ is the tangent space of $\X_1\times\X_2$ at $(x_1,x_2)$. Hence taking into account  the characterization  of $\wncRCD$ spaces in point $(iv)$ of Theorem \ref{thm:wncpropp} we deduce that
\begin{equation}
\label{eq:tensprod}
\vartheta_{N_1+N_2}[\X_1\times\X_2](x_1,x_2)=\vartheta_{N_1}[\X_1](x_1)\vartheta_{N_2}[\X_2](x_2)\qquad(\mm_1\times\mm_2)-a.e.\ (x_1,x_2).
\end{equation}
{ Hence if $\vartheta_{N_1}[\X_1](x_1)=c>0$ $\mm_1$-a.e.\ and $\vartheta_{N_2}[\X_2](x_2)=c^{-1}>0$ $\mm_2$-a.e.\ it trivially follows that $\vartheta_{N_1+N_2}[\X_1\times\X_2]=1$ a.e., thus showing that $\X_1\times\X_2$ is $\ncRCD$ (by $(iii)$ of Theorem \ref{thm:wncpropp} and \eqref{eq:stessadens}). Conversely, if the left-hand-side of \eqref{eq:tensprod} is a.e.\ equal to 1, then the identity \eqref{eq:tensprod} forces $\vartheta_{N_1}[\X_1]$ and $\vartheta_{N_2}[\X_2]$ to be a.e.\ constant and since the product of these constants must be 1 we must have $\vartheta_{N_1}[\X_1](x_1)=c$ $\mm_1$-a.e.\ and $\vartheta_{N_2}[\X_2](x_2)=c^{-1}>0$ $\mm_2$-a.e.\ for some $c>0$, which is the claim.}
\end{proof}

\section{Non-collapsed spaces}
\subsection{Continuity of $\HH^N$}\label{se:contvol}
In this section we prove the continuity of $\HH^N$ as stated in Theorem \ref{thm:contvol}.

 A key ingredient that we shall need is the ``almost splitting via excess theorem" proved  by Mondino and Naber in \cite[Theorem 5.1]{Mondino-Naber14}: we shall present such result in the simplified form that we need referring to \cite{Mondino-Naber14} for the more general statement.

Here and in the following for $p\in\X$ we put $\sfd_p(\cdot):=\sfd(\cdot,p)$ and for $p,q\in\X$ we put $e_{p,q}:=\sfd_p+\sfd_q-\sfd(p,q)$.
\begin{theorem}\label{thm:MN11} For every  $k\in \N$, \(N\in \R\), $1\le k\le N$ and   \(\varepsilon\in (0,1)\) there is  \(\delta_1=\delta_1(\varepsilon,k,N)\leq 1\) such that the following holds. 

Assume that  $(\X,\sfd,\mm)$ is an $\RCD(-\delta_1 ,N)$ space with $\supp(\mm)=\X$ and that there are points  \(\bar x ,\{p_i, q_i\}_{1\le i\le k}, \{p_i+p_j\}_{1\le i<j\le k}\) in $\X$ with \(\sfd(p_i,\bar x), \sfd(q_i,\bar x), \sfd(p_i+p_j,\bar x)\ge 1/\delta_1\) such that
\begin{equation}\label{eq:quasi} 
\sum_{1\le i\le k} \mean{B^\X_{R}(\bar x)} | \D e_{p_i,q_i} |^2\d \mm+\sum_{1\le i<j\le k} \mean {B^\X_{R}(\bar x)}\Big| \D\Big(\frac{\sfd_{p_i}+\sfd_{p_j}}{\sqrt 2}-\sfd_{p_i+p_j}\Big)\Big|^2\d \mm\le \delta_1
\end{equation} 
for all \(1\le R\le 1/\delta_1\). { 

Then  there exists a metric space \(\Y\) and a map \(\phi: \X\to \Y\)  such that if we define 
\[
u:=(\sfd_{p_1}-\sfd_{p_1}(\bar x ),\dots, \sfd_{p_k}-\sfd_{p_k}(\bar x)): \X\to\R^k, 
\]
the map \(U:=(u,\phi): \X\to \R^k\times \Y\)  provides an $\eps$-isometry of $B_1^{\X}(\bar x)$ to $B_1^{\R^k\times \Y}((0,\phi(x))$, i.e.:
\[
\begin{array}{lrr}
\forall x,y\in B_1^\X(\bar x)\text{ it holds }&\big|\sfd_{\X}(x,y)-\sfd_{\R^k\times \Y}(U(x),U(y))\big|&\!\!\!\leq \eps,\\
\forall z\in B_1^{\R^k\times \Y}((0,\phi(x))\text{ there is $x\in B_1^{\X}(\bar x)$ such that } &\sfd_{\R^k\times \Y}(U(x),z)&\!\!\!\leq\eps.
\end{array}
\]
Furthermore if \(k=N\) we can take \(\Y\) to be a single point.
}
\end{theorem}
Very shortly and roughly said, the idea of the proof is the following: For given $\eps>0$ one picks a sequence $\delta_{1,n}\downarrow0$ and a corresponding sequence of spaces $\X_n$ satisfying the assumptions for $\delta_1=\delta_{1,n}$ pmGH-converging to a limit $\X$. Then by an Ascoli-Arzel\`a-type argument, up to subsequences the corresponding functions $u_n:\X_n\to\R^k$ converge to a limit $u:\X\to\R^k$ and, this is the key point of the proof,  thanks to \eqref{eq:quasi} such limit map $u$ is a { metric submersion}. The compactness \eqref{eq:Gcomp} of the class of $\RCD(-1,N)$ spaces then gives the conclusion (see \cite[Theorem 5.1]{Mondino-Naber14} for the details).

An important  consequence of the above theorem is the following sort of `\(\varepsilon\)-regularity'  result (see also  \cite[Theorem 6.8]{Mondino-Naber14})  that we shall state for the case $k=N\in\N$ only; notice that, as discussed in \cite{Mondino-Naber14}, the map $u_\eps$ is  $(1+\eps)$-biLipschitz for arbitrary values of $k$, but in order to obtain the key inequality \eqref{eq:volest2} the `maximality' assumption $k=N\in\N$ is necessary (see in particular inequalities \eqref{eq:ddoubling} and \eqref{eq:step1}).

\begin{proposition}\label{thm:MN2} For every  $N\in \N$, $N\geq 1$ and   \(\varepsilon\in (0,1)\) there is  \(\delta_2=\delta_2(\varepsilon,N)>0\) such that the following holds. Let $(\X,\sfd,\mm)$ be  an $\RCD(-\delta_2,N)$ space with $\supp(\mm)=\X$ and $\bar x\in\X$ such that
\begin{equation}\label{viciniviciniGH}
\sfd_{\rm GH}\Big(\big(B^\X_{1/\delta_2}(\bar x),\sfd\big),  \big(B^{\R^N}_{1/\delta_2}(0),\sfd_{\Eu}\big)\Big)<\delta_2.
\end{equation}
Then there exists a Borel set  \(U_{\eps}\subset B_1^\X(\bar x)\)  and a $(1+\eps)$-biLipschitz map $u_\eps:U_\eps\to u_\eps(U_\eps)\subset \R^N $ such that 
\begin{equation}
\label{eq:volest2}
\LL^N\big( u_\eps(U_\eps)\cap B^{\R^N}_{1}(0)\big)\geq (1-\eps)\,\LL^N( B_1^{\R^N}(0)).
\end{equation}
\end{proposition}

%

\begin{proof} We divide the proof in two steps:\\
\noindent
    \underline{Step 1: construction of $U_\eps,\ u_\eps$ and $(1+\eps)$-biLipschitz estimate}  This is the content of \cite[Theorem 6.8]{Mondino-Naber14}, however since some aspects of the proof will be needed to obtain \eqref{eq:volest2} we briefly recall the argument.

    \underline{Step 1.1: basic ingredients} We start observing that for any $R\geq 1$ and { any \(f\in \Lip(\X)\)}, it holds the simple inequality
    \begin{equation}
\label{eq:grad2}
\int_{B_R(\bar x)}|\D f|^2\,\d\mm\leq\|f\|_{L^\infty(B_{2R}(\bar x))}\big(\|\Delta f\|_{L^1(B_{2R}(\bar x))}+\mm(B_{2R}(\bar x))\Lip(f)\big),
\end{equation}
as can be proved with an integration by parts (see e.g.\ \cite{Gigli12} for all the relevant definitions and properties of integration by parts and Laplacian) starting from $\int_{B_R(\bar x)}|\D f|^2\,\d\mm\leq\int_\X|\D f|^2\varphi\,\d\mm$ for   $\varphi:=(1-\sfd(\cdot,B_R(\bar x)))^+$. (In fact one can easily get rid of the term $\Lip(f)$ in the right hand side provided  $\varphi$ has bounded Laplacian. The existence of such cut-off functions - i.e.\ Lipschitz and with bounded Laplacian -  in the context of Ricci-limit spaces has been proved in \cite{Cheeger-Colding96} and frequently used as important technical tool in \cite{Cheeger-Colding97I,Cheeger-Colding97II,Cheeger-Colding97III,Colding96,Colding96b,Colding97}; their existence in the $\RCD$ setting has been proved in  \cite[Lemma 6.7]{AmbrosioMondinoSavare13-2} and \cite[Theorem 3.12]{Gigli-Mosconi14}, see also \cite[Lemma 3.1]{Mondino-Naber14}.).

A second ingredient is the Laplacian comparison estimate for the distance function (see \cite{Gigli12}) which ensures that on an $\RCD(-1,N)$ space, if $\sfd(p,\bar x)\geq 8R\geq 8$ then $\Delta\sfd_p\leq C(N,R)$ on $B_{4R}(\bar x)$ (this should be understood as an inequality between measures, but for the purpose of this outline let us think at $\Delta\sfd_p$ as a function). From this bound it is not hard to get the estimate 
\begin{equation}
\label{eq:deltal1}
\|\Delta\sfd_p\|_{L^1(B_{2R}(\bar x))}\leq C(N,R)\mm(B_{4R}(\bar x))\qquad\forall p\notin B_{8R}(\bar x)
\end{equation}
(this is in fact reverse engineering: in  \cite{Gigli12} the fact that $\Delta\sfd_p$ is a measure is obtained by proving an inequality like \eqref{eq:deltal1} with the total variation norm in place of the $L^1$ one).

    \underline{Step 1.2: geometric argument} Let $\tilde\eta,\delta_1\in(0,1)$ and notice - by direct simple computation - that there exists  $ R\geq \frac{32}{\delta_1} $ such that 
\[
\sup_{x\in B^{\R^N}_{8/\delta_1}(0)}\big(|x-Re_i|+|x+Re_i|-2R\big)\leq   \tilde\eta\qquad\forall i=1,\ldots,N.
\]
Hence if \eqref{viciniviciniGH} is satisfied for some $\delta_2\leq \min\{\tilde\eta,\frac1R\}$, letting $p_i,q_i\in\X$ be points corresponding to $ Re_i,- Re_i$ respectively in the $\delta_2$-isometry we obtain
\[
\|e_{p_i,q_i} \|_{L^\infty(B^{\X}_{{ 4/{\delta_1}}}(\bar x))}\leq 3\delta_2+\tilde\eta\leq 4\tilde\eta\qquad\forall i=1,\ldots,N.
\]
Noticing that $\sfd(p_i,\bar x),\sfd(q_i,\bar x)\geq  R-\delta_2\geq\frac{16}{\delta_1}$, from \eqref{eq:deltal1} we deduce that
\[
\|\Delta e_{p_i,q_i}\|_{L^1(B_{4/{\delta_1}}(\bar x))}\leq C(\delta_1,N)\mm(B_{8/\delta_1}(\bar x))\qquad\forall i=1,\ldots,N
\]
and since  $e_{p_i,q_i}$ is 2-Lipschitz, these last two bounds, \eqref{eq:grad2} and \eqref{e:BG} imply  that
\[
\mean{B_{ 2/\delta_1}}|\D e_{p_i,q_i}|^2\,\d\mm\leq  4\tilde\eta\, C(\delta_1,N).
\]
The same line of thoughts yields the analogous inequality for the function $\frac{\sfd_{p_i}+\sfd_{p_j}}{\sqrt 2}-\sfd_{p_i+p_j}$ for properly chosen points $p_i+p_j\in\X$ with $\sfd(p_i+p_j,\bar x)\geq 16/\delta_1$.

Now we fix $\eps\in(0,1)$, pick $\delta_1=\delta_1(\eps,N)$ given by Theorem \ref{thm:MN11}, let $\eta\ll\eps$ be a small parameter to be fixed later and notice that  we can rephrase what we just proved as:  there exists $\delta_2=\delta_2(\delta_1,\eta,N)\leq 1$ such that if \eqref{viciniviciniGH} is satisfied for such $\delta_2$ - which we shall hereafter assume -, then we can find points \(\{p_i, q_i\}_{1\le i\le N}, \{p_i+p_j\}_{1\le 1<j\le N}\) with \(\sfd(p_i,\bar x), \sfd(q_i,\bar x), \sfd(p_i+p_j,\bar x)\ge 2/\delta_1\) and such that
\begin{equation}\label{eq:quasi22} 
\sum_{1\le i\le N} \mean{B^\X_{2/\delta_1}(\bar x)} | \D e_{p_i,q_j} |^2\d \mm+\sum_{1\le i<j\le N} \mean {B^\X_{2/\delta_1}(\bar x)}\Big| \D\Big(\frac{\sfd_{p_i}+\sfd_{p_j}}{\sqrt 2}-\sfd_{p_i+p_j}\Big)\Big|^2\d \mm\le \eta^2.
\end{equation} 

    \underline{Step 1.3: use of the maximal function} Consider the function $f:\X\to\R$ defined as
\begin{equation}\label{eq:ff}
f(x):=\sum_{1\le i\le N}| \D e_{p_i,q_j} |^2+\sum_{1\le i<j\le N}\Big| \D\Big(\frac{\sfd_{p_i}+\sfd_{p_j}}{\sqrt 2}-\sfd_{p_i+p_j}\Big)\Big|^2
\end{equation}
and its maximal function $M: B_1^\X(\bar x)\to\R$ given by
\[
M(x)=\sup_{0<R<1/\delta_1} \mean{B_R^\X(x)} f(x)\d\mm.
\]
We put
\[
U:=\big\{x\in B_1^\X(\bar x):  M(x)\le \eta\big\}
\]
and note  that  the left hand side of \eqref{eq:quasi} is invariant under rescaling of the distance (essentially because  it holds \(|\D_r \sfd_p^r|=|\D\sfd_p|\), where  \(\sfd^r_p=\sfd_p/r \) and \(|\D_r\cdot|\) is the weak upper gradient computed with respect to the metric measure space \((\X,\sfd/r,\mm)\)).  Hence for  \(x\in U\) we can apply Theorem \ref{thm:MN11} { to the scaled space $(\X,\sfd/r,\mm)$ for $r\in(0,1)$ (notice that since we multiplied the distance by a factor $>1$, the space is `flatter' than the original one and in any case still  $\RCD(-\delta_2,N)$, in particular \ref{thm:MN11} is applicable) -} to infer that, provided \(\eta\), and hence \(\delta_2\),  are sufficiently small,  the map
\[
u:=(\sfd_{p_1}-\sfd_{p_1}(\bar x ),\dots, \sfd_{p_N}-\sfd_{p_N}(\bar x))
\]
is an $\eps$-isometry at every scale { in the range $(0,1)$} around points on $U$: this is sufficient to prove that it is \((1+\eps)\) bi-Lipischitz  when restricted to \(U\), see the proof of  \cite[Theorem 6.8]{Mondino-Naber14}  for the  details.


\medskip
\noindent
 \underline{Step 2: proof of estimate \eqref{eq:volest2}} From the trivial set identity
\[
B_1^{\R^N}(0)\cap u(U)=\Big(B_1^{\R^N}(0)\setminus\big(B_1^{\R^N}(0)\setminus u(B_1^\X(\bar x))\big)\Big)\setminus\big(u(B_1^\X(\bar x))\setminus u(U)\big)
\]
we deduce that
\begin{equation}
\label{eq:pezzi}
\mathcal L^N\big(B_1^{\R^N}(0)\cap u(U)\big)\geq \mathcal L^N\big(B_1^{\R^N}(0)\big)-\mathcal L^N\big(B_1^{\R^N}(0)\setminus u(B_1^\X(\bar x))\big)-\mathcal L^N\big(u(B_1^\X(\bar x))\setminus u(U)\big).
\end{equation}

\underline{Step 2.1: estimate of the size of $u(B_1^\X(\bar x))\setminus u(U)$.}Recall that since  \(\delta_2\le1\), the space \(\X\) is \(\RCD(-1,N)\)  and thus \(\mm\restr{ B^\X_{2/\delta_1}(\bar x)}\) is  doubling with a constant depending only on \(\delta_1$ and $N$. Hence according to the weak  \(L^1\) estimates for the  maximal function  we have
\begin{equation}
\label{eq:misU}
\frac{\mm(B^\X_{1}(\bar x)\setminus U)}{\mm(B_1^\X(\bar x))}\le \frac{C(\delta_1,N)}{\eta\,\mm(B_1^\X(\bar x))} \int_{B^\X_{2/\delta_1}(\bar x)} f(x)\,\d\mm \stackrel{\eqref{eq:quasi22}}\le C(\delta_1,N)\,\eta\frac{ \mm(B_{2/\delta_1}^\X(\bar x))}{\mm(B_1^\X(\bar x))}\le C(\delta_1,N) \,\eta.
\end{equation}
Now notice that    $u$ is $\sqrt N$-Lipschitz so that since  $u(B_1^\X(\bar x))\setminus u(U)\subset u(B_1^\X(\bar x)\setminus U)$ we have
\[
\mathcal L^N\big(u(B_1^\X(\bar x))\setminus u(U)\big)\leq \mathcal L^N\big(u(B_1^\X(\bar x)\setminus U)\big)\leq (\sqrt N)^N\HH^N\big(B_1^\X(\bar x)\setminus U\big)
\]
and therefore using \eqref{eq:hnm} with $R=1$ we get
\begin{equation}
\label{eq:pezzofacile}
\mathcal L^N\big(u(B_1^\X(\bar x))\setminus u(U)\big)\leq C(K,N)  \frac{\mm\big(B_1^\X(\bar x)\setminus U\big)}{\mm\big(B_1^\X(\bar x)\big)}\stackrel{\eqref{eq:misU}}\leq C(\delta_1,K,N)\,\eta.
\end{equation}

\underline{Step 2.2: estimate of the size of $B_1^{\R^N}(0)\setminus u(B_1^\X(\bar x))$.} {We divide this step into  further sub-steps \footnote{{We would like to thank Liming Yin for pointing out that this fix was needed in our original proof.}}:}

{
\underline{\emph{Step 2.2.1}}: We claim that for  any  \(\sigma \in (0,1)\) we can choose \(\tilde{\delta}\) such  that if \eqref{viciniviciniGH} is satisfied with \(\delta_{2}\le \tilde \delta\), then
\begin{equation}\label{e:boundaries}
u(B_{1}^{X}(\bar x)\setminus B_{1-\sigma/2}^{X}(\bar x))\subset B^{\R^{N}}_{1+\sigma}(0)\setminus B^{\R^{N}}_{1-\sigma}(0), 
\end{equation}
and 
\begin{equation}\label{e:dense}
\text{\(u(B_{1}^{X}(\bar x))\) is \(\sigma/2\sqrt{N}\)-dense in \(B^{\R^{N}}_{1}(0)\).}
\end{equation}
To prove the claim, we  argue by contradiction  and notice that if a sequence of spaces \(X_k\) satisfies  \ref{viciniviciniGH} with \(\delta_k\to 0\), then  the family of maps \(u_{k}\) constructed in Step 1 converges ``locally uniformly'' to an isometry \(u_{\infty}\) of  \(\mathbb R^{N}\) such that \(u_{\infty}(0)=0\). In particular \eqref{e:boundaries} and \eqref{e:dense} would be satisfied in the limit.
}

{\underline{\emph{Step 2.2.2}}:}  
{Since \(B_{1/2}^{\R^N}(\frac12 e_1)\subset B_{1}^{\R^N}(0)\cap B_1^{\R^N}(e_1)\) we have that
\begin{equation}
\label{eq:unmezzo}
\text{if \(S\subset B^{\mathbb \R^N}_{\sqrt{N}}(0)\) is \(\tfrac12\)-dense in \(B^{\mathbb \R^N}_1(0)\), then  \(S\cap B^{\mathbb \R^N}_1(e_1)\ne \emptyset\)}.
\end{equation}
Let $\bar \delta:=\delta_1(\tfrac12,N)$ where \(\delta_1(\tfrac12,N)\) is  given by Theorem \ref{thm:MN11} (with \(k=N\)). Put
\[
\lambda=\lambda(N):=\frac{\sqrt N+1}{\bar\delta}.
\]
Apply Lemma \ref{le:colding2} below to such $\lambda$ and to the open set $A:=B^{\R^N}_1(0)\setminus u\big({\bar B^\X_1(\bar x)}\big)$ (notice that since $u(\bar x)=0$ we have $A\neq B^{\R^N}_1(0)$) to find balls  \(\{B^{\R^N}_{r_k}(y_k)\}_{k=1}^M\) satisfying the properties $(i),\ldots,(iv)$ stated in the lemma. By property $(ii)$, for every \(k=1,\dots,M\)  there exists \(z_k\in {\bar B_1^{\X}(\bar x)}\) such that \(u(z_k)\in \partial B^{\R^N}_{r_k}(y_k)\). 
Moreover, by \eqref{e:dense}, since \(A\) can not contain ball of radius larger than \(\sigma/2\sqrt{N}\),  we also have  
\[
r_{k}\le \frac{\sigma}{2\sqrt{N}}.
\]
}

{\underline{\emph{Step 2.2.3}:}  We now claim that for every \(k=1,\dots,M\), either 
\begin{equation}\label{e:boundary-close}
\dist(z_{k},\partial B_{1}^{X}(\bar x))\le \frac{\sigma}{2},
\end{equation}
or there exists a radius \(\rho_k\in (r_k, r_k/ \bar \delta)\) such that 
\begin{equation}\label{olee}
\mean{B_{\rho_k}(z_k)} f(x) \d \mm\ge  \bar \delta.
\end{equation}
Here \(f\) is the function defined in~\eqref{eq:ff} (recall that the points \(p_i, q_i, p_i+p_j\) have been already fixed in Step 1.2) and  \(\bar \delta=\delta_1(\tfrac12,N)\) is  given by Theorem \ref{thm:MN11}. Assume indeed that \eqref{olee} and \eqref{e:boundary-close} both fail. The failure of  \eqref{e:boundary-close} with   \eqref{e:boundaries}  implies that \(u(z_{k})\in B^{\R^{n}}_{1}\setminus B^{\R^{N}}_{1-\sigma}\) and thus,  exploiting the  \(\sqrt{N}\)-Lipschitzianity of \(u\)   and that  \(r_{k}\le\sigma/2 \sqrt{N}\), we get that 
\[
B^\X_{r_k}(z_k)\subset B^{X}_{1-\sigma/2+\sigma/2\sqrt N}(\bar x)\subset B^{X}_{1}(\bar x)
\]
and
\[
u\big(B^\X_{r_k}(z_k)\big) \subset B^{\R^N}_{\sqrt{N} r_k}(u(z_k)) \subset B_{1-\sigma +\sqrt{N}r_{k}}^{\R^{N}} \subset B_{1}^{\R^{N}}.
\]
By point (i) in Lemma \ref{le:colding2}, 
\[
B^{\R^N}_{r_k} (y_k)\subset B_{1}^{\R^{N}}\setminus u(B^{X}_{1}(\bar x)).
\]
Combining this with the two previous inclusions we deduce that  
\begin{equation}\label{eq:contrr}
u\big(B^\X_{r_k}(z_k)\big)\subset   \Big(B^{\R^N}_{\sqrt{N} r_k}(u(z_k))\setminus B^{\R^N}_{r_k} (y_k)\Big).
\end{equation}
The failure of \eqref{olee} allows to apply the scaled version of  Theorem~\ref{thm:MN11} to   deduce that  $u(B^\X_{r_k} (z_k))$ is \(\tfrac12 r_k\)-dense in $B^{\R^N}_{r_k}(u(z_k))$, a contradiction with (the scaled version of) \eqref{eq:unmezzo} and  \eqref{eq:contrr}.
}

{We now partition  \(\{1,\dots, M\}\) into two sets: 
\[
\mathcal B=\Big\{k\text{ such that \eqref{e:boundary-close} holds}\Big\}\qquad \mathcal I=\Big\{k:\text{ such that \eqref{olee} holds}\Big\},
\]
and aim to estimate  \(\sum_{k\in \mathcal I} r_{k}^{N}\) and  \(\sum_{k\in \mathcal B} r_{k}^{N}\).
}

{\underline{\emph{Step 2.2.4}, Estimate on  \(\sum_{k\in \mathcal I} r_{k}^{N}\):}  Using  that $\mm\restr{B_{2/\bar \delta}(\bar x)}$ is doubling with a constant depending only on $\bar \delta$ and $N$, and that $\rho_k\leq r_k/\bar\delta\leq 1/\bar\delta$ we see that
\begin{equation}
\label{eq:ddoubling}
\mm\big(B^\X_{2/\bar \delta \delta}(\bar x)\big)\leq C(\bar \delta,N)\frac{\mm(B^\X_{\rho_k}(z_k))}{\rho_k^N}\qquad\text{ for every }k\in \mathcal I
\end{equation}
and therefore, using that \(r_{k}\le \rho_{k}\):
\begin{equation}
\label{eq:step1}
\sum_{k\in \mathcal I} r_{k}^{N}\le  \sum_{k\in \mathcal I}\rho_k^N\stackrel{\eqref{eq:ddoubling}}\le \frac{C(\bar \delta,N)}{\mm(B^\X_{2/\bar \delta}(\bar x))} \sum_{k\in \mathcal I}\mm(B_{\rho_k}^\X (z_k))\stackrel{\eqref{olee}}\le  \frac{C(\bar\delta,N)}{\bar\delta\,\mm(B^\X_{2/\bar \delta}(\bar x))} \sum _{k\in \mathcal I} \int_{B^\X_{\rho_k} (z_k)} f\,\d\mm.
\end{equation}
Since \(u\) is \(\sqrt{N}\)-Lipschitz and  \(u(z_k)\in \partial B^{\R^N}_{r_k}(y_k)\) we have
\[
u(B^\X_{\rho_k}(z_k))\subset u(B^\X_{r_k/\bar\delta}(z_k))\subset B^{\R^N}_{\sqrt{N}r_k/\bar\delta}(u(z_k))\subset B^{\R^N}_{(\sqrt{N}+1)r_k/\bar\delta}(y_k)=B^{\R^N}_{\lambda(N) r_k}(y_k).
\]
This inclusion and the fact that, by property $(iii)$ in Lemma \ref{le:colding2}, the balls \(B^{\R^N}_{\lambda(N) r_k}(y_k)\)  are disjoint, grant that  the balls \(B^\X_{\rho_k}(z_k)\subset B^\X_{2/\bar \delta}(\bar x) \) are disjoint as well.  Hence from \eqref{eq:step1} we get
\begin{equation}
\label{e:B}
\sum_{k\in \mathcal I} r_k^N\leq \frac{C(\delta_1,N)}{\bar\delta\,\mm(B^\X_{2/\bar \delta}(\bar x))}   \int_{B^\X_{2/\bar \delta} (\bar x)} f\,\d\mm\stackrel{\eqref{eq:quasi22}}\leq \frac{C(\bar \delta,N)}\eta^2.
\end{equation}
}

{\underline{\emph{Step 2.2.5}, Estimate on  \(\sum_{k\in \mathcal B} r_{k}^{N}\):} Thanks to \eqref{e:boundaries} and \eqref{e:boundary-close} we get that \(\dist(y_{k} , \partial B_{1}^{\R^{N}})\le \sigma\). In particular the balls \(B^{\R^N}_{\lambda(N) r_k}(y_k)\) are disjoint and, since \(r_{k}\le \sigma/\sqrt{N}\),  contained in \(C(N)\sigma\) neighborhood of \(\partial B_{1}\). Hence 
\begin{equation}
\label{e:I}
\sum_{k\in \mathcal I} r_k^N \le C(N)\sigma 
\end{equation}
}

{\underline{\emph{Step 2.2.6}, Conclusion:} 
Combining \eqref{e:I} and \eqref{e:B} with  property $(iv)$ in Lemma \ref{le:colding2}  we get 
\[
\mathcal L^N\big(B^{\mathbb R^N}_1\setminus u(\overline{B^X_1(\bar x)}\big)\leq C(N)\sum_{k\in \mathcal I\cup \mathcal B} r_k^N \le C(N, \bar \delta)(\eta+\sigma).
\]
The conclusion comes plugging this bound and \eqref{eq:pezzofacile} into \eqref{eq:pezzi} and  by picking $\eta$ and \(\sigma\), and thus $\delta_2$, sufficiently small w.r.t.\ $\eps$.}
\end{proof}

In the proof of the above proposition we have used the following  covering {Lemma from ~\cite{Colding97}, we report here its simple proof for the sake of completness}.
\begin{lemma}\label{le:colding2}Let $N\in\N$, $N\geq 1$, and $\lambda\geq 1$. Then there exists a constant $C=C(\lambda,N)$ such that the following holds. For every  \(A\subsetneq B_1^{\R^N}(0)\) open  there exists a finite family of balls \(\{B^{\R^N}_{r_k}(y_k)\}_{k=1}^M\),  such that
\begin{itemize}
\item[(i)] \(B^{\R^N}_{r_k}(y_k)\subset A\) for every $k$,
\item[(ii)] $\big(\partial{B^{\R^N}_{r_k}(y_k)}\cap \partial A\big)\setminus\partial B_1^{\R^N}(0)\ne \emptyset\) for every $k$,
\item[(iii)]The family  \( \{B^{\R^N}_{\lambda r_k}(y_k)\}_{k=1}^M\) is disjoint,
\item[(iv)] It holds
\begin{equation}\label{fine}
\LL^N(A)\le C(\lambda, N)\sum_{k=1}^M r_k^N.
\end{equation}
\end{itemize}
\end{lemma}
\begin{proof} We claim that if $B\subset A$ is a ball with $\overline B\subset B_1(0)$, then there exists another ball $B'$ such that
\begin{equation}
\label{eq:bp}
B\subset B'\subset A\qquad\text{ and }\qquad \big(\partial {B'}\cap\partial A\big)\setminus\partial B_1^{\R^N}(0)\ne \emptyset.
\end{equation}
Indeed, let $B=B_r(v)$, put $B_t:=B_{(1-t)r+t}(tv)$, notice that the family $(0,1)\ni t\mapsto B_t$ is increasing and that $\overline B_t\subset B_1(0)$ for every $t\in[0,1)$. Since $\cup_{t\in[0,1)}B_t=B_1(0)$ and $A$ is strictly included in $B_1(0)$,  by a simple compactness argument we find a least $t_0\in[0,1)$ such that $\overline B_{t_0}\cap (\R^N\setminus A)\neq\emptyset$ and since $\overline B_{t_0}\subset B_1(0)$ we must also have $\overline B_{t_0}\cap (B_1(0)\setminus A)\neq\emptyset$. To conclude the proof of the claim simply notice that from the trivial identity $B_{t_0}=\cup_{t\in[0,t_0)}\overline B_t$ and the minimality of $t_0$ we have $B_{t_0}\subset A$, hence $B':=B_{t_0}$ does the job. 

Now let $K\subset A$ be compact so that $\mathcal L^N(A)\leq 2\mathcal L^N(K)$ and, by compactness, find a finite family of balls $\{B_{r_j}^{\R^N}(y_j)\}_{j=1}^L$ with closure included in $A$ and covering $K$. Up to replace each of these balls with the corresponding one $B'$ as in \eqref{eq:bp}, we can assume  that this family satisfies $(ii),(iii)$.

We shall now build a subfamily for which  $(iii),(iv)$ also hold. Up to reordering we can assume that  $r_1\geq \ldots\geq r_L$.
Then put $j_1:=1$ and define recursively 
\[
j_n:=\text{least index $j$ such that }B_{\lambda r_j}^{\R^N}(y_j)\cap \bigcup_{i=1}^{n-1}B_{\lambda r_{j_i}}^{\R^N}(y_{j_i})=\emptyset.
\]
Since the original family was finite, this process ends at some step $M$ and, by construction, the family $\{B_{r_{j_i}}^{\R^N}(y_{j_i})\}_{j=1}^M$ fulfils $(i),(ii),(iii)$. Also, by construction for every $j>1$ there is $i\in\{1,\ldots,M\}$ such that $r_{j_i}\geq r_j$ and $B_{\lambda r_j}^{\R^N}(y_{j})\cap B_{\lambda r_{j_i}}^{\R^N}(y_{j_i})\neq\emptyset$, hence $B_{\lambda r_j}^{\R^N}(y_{j})\subset B_{3\lambda r_{j_i}}^{\R^N}(y_{j_i})$ and thus
\[
K\subset \bigcup_{j=1}^LB_{r_j}^{\R^N}(y_j)\subset \bigcup_{j=1}^LB_{\lambda r_j}^{\R^N}(y_j)\subset \bigcup_{i=1}^MB_{3\lambda r_{j_i}}^{\R^N}(y_{j_i}),
\]
so that \eqref{fine} holds with $C(\lambda,N):=2\omega_N(3\lambda)^N$.
\end{proof}
We can now prove the continuity of $\HH^N$ under a uniform Riemannian-curvature-dimension condition:
\begin{proofb}{\bf of Theorem \ref{thm:contvol}}\\
\noindent\underline{Set up}  If $N\notin\N$, Theorem \ref{thm:rett} implies that $\HH^N(\X)=0$ for any $\RCD(K,N)$ space $\X$, hence in this case there is nothing to prove. We shall therefore assume $N\in\N$.

 Let $(Z_n)\subset\mathbb B_{K,N,R}$ be GH-converging to some limit $Z_\infty\in{\mathbb B}_{K,N,R}$ and for each $n\in\N$, let $(\X_n,\sfd_n,\mm_n,x_n)$ be an $\RCD(K,N)$ space such that $Z_n=\bar B^{\X_n}_R(x_n)$. Up to replace $\mm_n$ with $\mm_n/\mm_n(Z_n)$ we can, and will, assume that $\mm_n(Z_n)=1$ for every $n\in\N$. Then by the compactness of the class of $\RCD(K,N)$ spaces \eqref{eq:Gcomp}, up to pass to a subsequence, not relabeled, we  have  $(\X_n,\sfd_n,\mm_n,x_n)\stackrel{pmGH}\to (\X_\infty,\sfd_\infty,\mm_\infty,x_\infty)$ for some pointed $\RCD(K,N)$  space $\X_\infty$. It is then clear  that $Z_\infty=\bar B_R^{\X_\infty}(x_\infty)$. To conclude it is now sufficient to prove that $\HH^N(Z_n)\to\HH^N(Z_\infty)$, because such continuity property is independent on the subsequence chosen. 

\noindent\underline{Upper semicontinuity} Let $(\X,\sfd,\mm)$ be a generic  $\RCD(K,N)$ space  with $\supp(\mm)=\X$. We claim that
\begin{equation}
\label{eq:mhn}
\HH^N=\frac1{\vartheta_N[\X]}\mm
\end{equation}
where it is intended that $\frac1{\vartheta_N[\X]}(x)=0$ if $\vartheta_N[\X](x)=\infty$. To see this start observing that the Bishop-Gromov inequality \eqref{e:BG} grants that 
\begin{equation}
\label{eq:bbasso}
\vartheta_N[\X](x)\geq C(K, N) \mm(B_1(x)),
\end{equation}
so that $\frac1{\vartheta_N[\X]}\in L^1_{loc}(\X,\mm)$ and the right hand side of \eqref{eq:mhn} defines a Radon measure. Then, in the notation of Theorem \ref{thm:rett}, by \eqref{eq:densk} for $k=N$ it follows that equality holds in \eqref{eq:mhn} if we restrict both sides to $\cup_jU^N_j$, so that to conclude it is sufficient to show that on  $\X\setminus (\cup_jU^N_j)$ both sides of \eqref{eq:mhn} are zero. The fact that $\HH^N(\X\setminus (\cup_jU^N_j))=0$ is a trivial consequence of Proposition \ref{prop:basehn} and Theorem \ref{thm:rett}, while \eqref{eq:mac} and \eqref{eq:densk} imply that for $k<N$, $k\in\N$ we have $\vartheta_N[\X]=\infty$ $\mm$-a.e.\ on $\cup_jU_j^k$. Hence our claim \eqref{eq:mhn} is proved.

Now we apply Lemma \ref{le:fatou} to the functions 
\[
f_n:=-\frac1{\vartheta_N[\X_n]}\nchi_{\bar B_R^{\X_n}(x_n)},\qquad n\in\N\cup\{\infty\},
\]
 point $(i)$ of Lemma \ref{lm:propt} grants that \eqref{eq:glimi} is satisfied, while from \eqref{eq:bbasso} and the assumption $\mm_n(Z_n)=1$ it easily follows that the $f_n$'s are uniformly bounded from below by a continuous function with bounded support. Hence the conclusion of Lemma \ref{le:fatou} gives
\[
-\HH^N(Z_\infty)\stackrel{\eqref{eq:mhn}}=\int f_\infty\,\d\mm_\infty\leq\limi_{n\to\infty}\int f_n\,\d\mm_n\stackrel{\eqref{eq:mhn}}=\limi_{n\to\infty}-\HH^N(Z_n)
\]
and thus the desired upper semicontinuity:
\[
\lims_{n\to\infty}\HH^N(Z_n)\le \HH^N(Z_\infty).
\]
\noindent\underline{Lower semicontinuity} Theorem \ref{thm:rett} ensures that for $\HH^N$-a.e.\ $x\in Z_\infty$ the metric tangent space of $\X_\infty$ at $x$ is $\R^N$.

Now fix $\eps>0$, let $\delta_2:=\delta_2(\eps,N)$ be given by Proposition \ref{thm:MN2} and notice that what we just said grants that for $\HH^N$-a.e.\ $x\in Z_\infty$ there exists $\bar r=\bar r(x)$ such that for every $r\in(0,\bar r)$ we have 
\[
\sfd_{\rm GH}(B^{\X_\infty}_{r/\delta_2}(x),B^{\R^N}_{r/\delta_2}(0))\leq r\delta_2,
\]
thus, since Proposition \ref{prop:basehn} grants that $\HH^N(Z_\infty)<\infty$, we can use Vitali's covering lemma (see e.g.\ \cite[Theorem 2.2.2]{AmbrosioTilli04}) to find a finite number of points $y_{\infty,i}\in B^{\X_\infty}_1(x_\infty)$ and radii $r_i>0$, $i=1,\ldots,M$, such that
\begin{subequations}
\begin{align}
\label{eq:matt1}
\sfd_{\rm GH}(B^{\X_\infty}_{r_i/\delta_2}(y_{\infty,i}),B^{\R^N}_{r_i/\delta_2}(0))&\leq r_i\delta_2\qquad&&\forall i,\\
\label{eq:matt2}
\bar B^{\X^\infty}_{r_i}(y_{\infty,i})\cap\bar B^{\X^\infty}_{r_{i'}}(y_{\infty,i'})&=\emptyset \qquad&&\forall i\neq i',\\
\label{eq:matt3}
\sfd_\infty(y_{\infty,i},x_\infty)+r_i&<1\qquad&&\forall i,\\
\label{eq:matt4}
\HH^N(B^{\X_\infty}_1(x_\infty))&\leq\eps+\omega_N\sum_ir_i^N.
\end{align}
\end{subequations}
For each $i$ find a sequence $y_{n,i}\stackrel{GH}\to y_{\infty,i}$ (recall \eqref{eq:exseq}) and notice that there is $\bar n\in \N$ such that for every $n\geq \bar n$ properties \eqref{eq:matt1}, \eqref{eq:matt2}, \eqref{eq:matt3} hold with $y_{n,i}$ and $x_n$ in place of $y_{\infty,i},x_\infty$ respectively for any $i$. In particular, from \eqref{eq:matt1} for $y_{n,i}$ we can apply the scaled version of Proposition \ref{thm:MN2} to deduce that
\begin{equation}
\label{eq:matt5}
\HH^N(B^{\X_n}_{r_i}(y_{n,i}))\geq (1-\eps)\omega_Nr_i^N\qquad\forall i
\end{equation}
and since  \eqref{eq:matt2}, \eqref{eq:matt3} for the $y_{n,i}$'s ensure that the balls $B^{\X_n}_{r_i}(y_{n,i})$, $i=1,\ldots,M$, are disjoint and contained in $B_1^{\X_n}(x_n)$ for every $n\geq\bar n$ we deduce that
\[
\begin{split}
\HH^N(B_1^{\X_n}(x_n))\geq\sum_i\HH^N(B^{\X_n}_{r_i}(y_{n,i}))\stackrel{\eqref{eq:matt5}}\geq (1-\eps)\omega_N\sum_ir_i^N\qquad\forall n\geq \bar n.
\end{split}
\]
Hence  from \eqref{eq:matt4} we obtain
\[
\limi_{n\to\infty}\HH^N(B_1^{\X_n}(x_n))\geq (1-\eps)\big(\HH^N(B^{\X_\infty}_1(x_\infty))-\eps\big)
\]
and, recalling \eqref{eq:hnsfere}, we conclude by the arbitrariness of $\eps>0$.
\end{proofb}
\subsection{Dimension gap}
In this section we shall prove the `dimension gap' Theorem \ref{thm:dimgap} and in doing so we will closely follows the arguments in \cite[Section 3]{Cheeger-Colding97I}. 

The crucial part of argument, provided in  Proposition \ref{le:nm1}, is the proof that the Hausdorff dimension of the set of points for which no tangent space splits off a line is at most $N-1$. To clarify the structure of the proof, we isolate in the following lemma the measure-theoretic argument which will ultimately lead to  such estimate on the dimension (see \cite[Proposition 3.2]{Cheeger-Colding97I}):
\begin{lemma}\label{le:figo}
Let $\X$ be an $\RCD(K,N)$ space and $\Omega_\eta\subset \X$  Borel subsets indexed by a parameter $\eta> 0$ such that  it holds
\begin{equation}
\label{eq:hypmes}
\mm(\Omega_\eta)\leq C\eta\qquad \forall\eta\in(0,c)
\end{equation}
for some $C,c>0$. For $\tau\in(0,1)$  consider the sets
\begin{equation}
\label{eq:omegatau}
\Omega_{\tau,\eta}:=\{x\in\Omega_\eta\ :\ \sfd(x,\X\setminus\Omega_\eta)>\tau\eta\}.
\end{equation}
Then
\[
\dim_\HH\Big(\bigcap_{j\in\N}\bigcup_{k\geq j}\Omega_{\tau,2^{-k}}\Big)\leq N-1.
\]
\end{lemma}
\begin{proof} {For given $\bar x\in\X$, $R>0$ let $\Omega_{R,\eta}:=\Omega_\eta\cap B_R(\bar x)$ and $\Omega_{\tau,R,\eta}$ be defined as in \eqref{eq:omegatau} with $\Omega_{R,\eta}$ in place of \(\Omega_\eta\). Then for every $\tau,\eta\in(0,1)$ the definition easily gives $\Omega_{\tau,\eta}\cap B_{R-1}(\bar x)\subset \Omega_{\tau,R,\eta}$ and therefore 
\[
\dim_\HH\Big(\bigcap_{j\in\N}\bigcup_{k\geq j}\Omega_{\tau,2^{-k}}\Big)=\lim_{R\to\infty}\dim_\HH\Big(\bigcap_{j\in\N}\bigcup_{k\geq j}\Omega_{\tau,2^{-k}}\cap B_{R-1}(\bar x)\Big)\leq\lim_{R\to\infty}\dim_\HH\Big(\bigcap_{j\in\N}\bigcup_{k\geq j}\Omega_{\tau,R,2^{-k}}\Big).
\]
Hence }up to replacing $\Omega_\eta$ with $\Omega_\eta\cap B_R(\bar x)$ and then sending $R\uparrow\infty$  we can, and will, assume that $\Omega_\eta\subset B_R(\bar x)$ for every $\eta\in(0,1)$.

Now let $x_1,\ldots,x_n\in \Omega_{\tau,\eta}$ be with $\sfd(x_i,x_j)\geq \tau\eta$ and denote by $C'=C(K,N,R,\bar x)$ a constant depending only on $K,N,R,\bar x$ (and thus independent on  $\tau,\eta$) whose values might change in the various instances it appears: since the balls $B_{\tau\eta/2}(x_i)$ are disjoint and  contained in $\Omega_\eta$ we have
\[
\begin{split}
n&\stackrel{\eqref{e:BG}}\leq C'\sum_{i=1}^n\mm(B_1(x_i))\stackrel{\eqref{e:BG}}\leq C'(\tau\eta)^{-N} \mm\Big(\bigcup_{i=1}^nB_{\frac{\tau\eta}2}(x_i)\Big)\leq C' (\tau\eta)^{-N}\mm( \Omega_\eta )\stackrel{\eqref{eq:hypmes}}\leq  CC'(\tau\eta)^{-N}\eta.
\end{split}
\]
If the family $\{x_1,\ldots,x_n\}$ is maximal we have $ \Omega_{\tau,\eta}\subset \bigcup_{i=1}^n B_{\tau\eta }(x_i)$ and thus for any $\eps>0$ the above implies
\begin{equation}
\label{eq:heps}
\HH^{N-1+\eps}_{2\tau\eta}( \Omega_{\tau,\eta})\leq CC'(\tau\eta)^{-N}\eta (2\tau\eta)^{N-1+\eps}\leq CC'\tau^{\eps-1}\eta^\eps.
\end{equation}
Hence for any $j\in\N$ we have
\[
\begin{split}
\HH^{N-1+\eps}_{2\tau 2^{-j}}\Big(\bigcap_{j'\in\N}\bigcup_{k\geq j'}\Omega_{\tau,2^{-k}}\Big)&\leq \HH^{N-1+\eps}_{2\tau 2^{-j}}\Big(\bigcup_{k\geq j}\Omega_{\tau,2^{-k}}\Big)\leq\sum_{k\geq j}\HH^{N-1+\eps}_{2\tau 2^{-k}}(\Omega_{\tau,2^{-k}})\\
\text{by \eqref{eq:heps}}\qquad\qquad&\leq \sum_{k\geq j} CC'\tau^{\eps-1}2^{-\eps k}=CC'\tau^{\eps-1}2^{\eps(1-j)}
\end{split}
\]
and letting $j\uparrow\infty$ we conclude.
\end{proof}
Let us now give few definitions, following \cite[Definition 2.10]{Cheeger-Colding97I}.  For a given $\RCD(K,N)$ space $\X$ with measure having full support we define
\[
\mathcal E_1(\X):=\big\{x\in\X\ :\ \text{every metric tangent space at $x$ splits off a line}\big\}
\]  
where we say that a metric measure space $(\X,\sfd,\mm)$ splits off a line provided $\X$ is isomorphic to the product of another  metric measure space $(\X',\sfd',\mm')$ and the real line, i.e.\ there is a measure preserving isometry $\Phi:\X'\times\R\to\X$, where the measure on $\X'\times\R$ is the product of $\mm'$ and the Lebesgue measure and the distance is given by
\[
\sfd^2_{\X'\times\R}\big((x,t),(y,s)\big):=\sfd'(x,y)^2+|t-s|^2,\qquad\forall x,y\in\X',\ t,s\in\R.
\]
Also, for given $\bar x\in\X$ we put
\[
\begin{split}
\mathcal E_{1,\bar x}(\X):=\left\{\begin{array}{rl}
&\text{$\forall \eps,\eps'>0$ there exists $\bar r=\bar r(\eps,\eps',x)$ such that for every $r\in(0,\bar r)$}\\
x\neq\bar x\ :&\text{there is $y\in B_{\eps r}(x)$ and a unit speed geodesic}\\
&\text{$\gamma:[0,\sfd(\bar x,y)+\tfrac{r}{\eps'}]\to\X$ such that $\gamma_0=\bar x$ and $\gamma_{\sfd(\bar x,y)}=y$}
\end{array}\right\}.
\end{split}
\]
{We claim that
\begin{equation}
\label{eq:trivialE}
\mathcal E_{1,\bar x}(\X)\subset \mathcal E_{1}(\X)\qquad\forall \bar x\in\X.
\end{equation}
Indeed, let $x\in \mathcal E_{1,\bar x}(\X)$, $\eps,\eps'>0$ and $\bar r=\bar r(\eps,\eps',x)$ as above. Then for $r\in(0,\bar r)$ let  $y,\gamma$ as above and notice that the appropriate restriction of $\gamma$ is a geodesic of length $\geq 2\min\{\frac r{\eps'},\sfd(x,\bar x)-\eps r\}$ whose midpoint $y$ has distance $\leq\eps r$ from $x$. Hence in the rescaled space $(\X_r,\sfd_r):=(\X,\sfd/r)$ there is a geodesic of length $\geq 2\min\{\frac 1{\eps'},\frac{\sfd(x,\bar x)}r-\eps \}$ whose midpoint has distance $\leq\eps$ from $x$. Letting $r\downarrow0$ we conclude by a compactness argument that on every tangent space at $x$ there is a geodesic of length $\geq \frac2{\eps'}$ whose midpoint has distance $\leq \eps$ from the origin, thus the arbitrariness of $\eps,\eps'$ and again a compactness argument ensure the existence of a line through the origin. Since every tangent space to an $\RCD(K,N)$ space is a $\RCD(0,N)$ space, the splitting theorem \cite{Gigli13}, \cite{Gigli13over}  gives \eqref{eq:trivialE}.} 

We then have the following:
\begin{proposition}\label{le:nm1}
Let $(\X,\sfd,\mm)$ be an $\RCD(K,N)$ space and $\bar x\in \X$. Then
\begin{equation}
\label{eq:dimest}
\dim_\HH(\X\setminus \mathcal E_{1,\bar x}(\X))\leq N-1
\end{equation}
and thus also
\begin{equation}
\label{eq:dimest2}
\dim_\HH(\X\setminus \mathcal E_{1}(\X))\leq N-1.
\end{equation}
\end{proposition}
\begin{proof}

\noindent\underline{Step 1: structure of the argument.} From \eqref{eq:trivialE} and \eqref{eq:dimest} the estimate \eqref{eq:dimest2} follows, hence we focus in proving \eqref{eq:dimest}. Since trivially $\dim_\HH(\{\bar x\})=0\leq N-1$, to conclude it is sufficient to prove that for any $R>2$ we have 
\begin{equation}
\label{eq:dimest3}
\dim_\HH\Big({\rm Ann}_{R/2}(\bar x)\cap \big(\X\setminus \mathcal E_{1,\bar x}(\X)\big)\Big)\leq N-1\qquad\text{where}\qquad {\rm Ann}_R(\bar x):=B_R(\bar x)\setminus B_{1/R}(\bar x).
\end{equation}
Fix $\bar x\in X$ and for $\eta>0$ define
\[
\begin{split}
\Omega^R_\eta:=\{x\in{\rm Ann}_{R}(\bar x)\ :\  \text{there is no unit speed geodesic } & \gamma:[0,\sfd(x,\bar x)+\eta]\to \X\\
&\text{such that $\gamma_0=\bar x$, $\gamma_{\sfd(x,\bar x)}=x$}\}
\end{split}
\]
and, for $\tau\in(0,1)$, define $\Omega^R_{\tau,\eta}$ as in \eqref{eq:omegatau}. We shall prove that for any $R>2$ we have
\begin{equation}
\label{eq:1}
\mm(\Omega^R_\eta)\leq  C\eta \qquad\forall \eta\in(0,\tfrac{1}{R})
\end{equation}
for some $C=C(K,N,R,\bar x)$ and
\begin{equation}
\label{eq:2}
{\rm Ann}_{R/2}(\bar x)\cap(\X\setminus\mathcal E_{1,\bar x})\subset\bigcup_{i}\Big( {\rm Ann}_{R/2}(\bar x)\cap  \bigcap_j\bigcup_{k\geq j}\Omega^R_{2^{-i},2^{-k}}\Big).
\end{equation}
Thanks to Lemma \ref{le:figo}, these are sufficient to get \eqref{eq:dimest3} and the conclusion.

\noindent\underline{Step 2: proof of \eqref{eq:1}} We assume $\mm({\rm Ann}_{R}(\bar x))>0$ or there is nothing to prove, then we put $\mu_0:=\mm({\rm Ann}_{R}(\bar x))^{-1}\mm\restr{{\rm Ann}_{R}(\bar x)}$, 
 $\mu_1:=\delta_{\bar x}$ and let $\ppi\in \prob{C([0,1],\X)}$ be the only optimal geodesic plan from $\mu_0$ to $\mu_1$ (see \cite{GigliRajalaSturm13}). Then from \cite[Theorem 3.4]{GigliRajalaSturm13} we know that $(\e_t)_*\ppi\ll\mm$ for every $t\in[0,1)$ and that for its density $\rho_t$ it holds
\[
\rho_t^{-\frac1N}(\gamma_t)\geq (\mm({\rm Ann}_{R}(\bar x)))^{\frac1N}\sigma_{K,N}^{(1-t)}(\sfd(\gamma_0,\gamma_1))\qquad\text{where}\qquad\sigma_{K,N}^{(t)}(d):=\frac{\sinh(td\sqrt{|K|N})}{\sinh(d\sqrt{|K|/N})}
\]
$\ppi$-a.e.\ $\gamma$, hence using the fact that $\sigma_{K,N}^{(t)}(d)$ is decreasing in $d$ we deduce that
\begin{equation}
\label{eq:denst}
(\e_t)_*\ppi\leq\frac1{\mm({\rm Ann}_{R}(\bar x))\big(\sigma_{K,N}^{(1-t)}(R)\big)^N}\mm
\end{equation}
while the construction ensures that
\begin{equation}
\label{eq:conct}
(\e_t)_*\ppi\text{ is concentrated on $B_{R}(\bar x)\setminus\Omega^R_{\frac{t}{R}}$ }\qquad\forall t\in(0,1).
\end{equation}
Therefore for $\eta< \frac1{ R}$ using the above with $t:=\eta R$ we have
\begin{equation}
\label{eq:alto}
\begin{split}
\mm\big(B_{R}(\bar x)\setminus\Omega^R_\eta\big)&\stackrel{\eqref{eq:denst}}\geq {\mm({\rm Ann}_{R}(\bar x))}{\big(\sigma_{K,N}^{(1-\eta R)}(R)\big)^N}\,\ppi\Big(\e_{\eta R}^{-1}\big( B_{R}(\bar x)\setminus\Omega^R_\eta\big)\Big)\\
&\stackrel{\eqref{eq:conct}}= {\mm({\rm Ann}_{R}(\bar x))}{\big(\sigma_{K,N}^{(1-\eta R)}(R)\big)^N}
\end{split}
\end{equation}
and since $\Omega^R_\eta\subset{\rm Ann}_{R}(\bar x)\subset B_R(\bar x)$ yields $\Omega^R_\eta={\rm Ann}_{R}(\bar x)\setminus(B_{R}(\bar x)\setminus\Omega^R_\eta)$ in turn this gives
\[
\mm(\Omega^R_\eta)\stackrel{\eqref{eq:alto}}\leq \mm({\rm Ann}_{R}(\bar x))\Big(1-\big(\sigma_{K,N}^{(1-\eta R)}(R)\big)^{N}\Big)
\]
which, using the explicit expression of $\sigma_{K,N}^{(1-\eta R)}(R)$, gives   our claim \eqref{eq:1}.

\noindent\underline{Step 3: proof of \eqref{eq:2}} We shall prove the equivalent inclusion
\begin{equation}
\label{eq:2b}
{\rm Ann}_{R/2}(\bar x)\cap \mathcal E_{1,\bar x}\supset{\rm Ann}_{R/2}(\bar x)\cap  \bigcap_{i}  \bigcup_j\bigcap_{k\geq j}\big(\X\setminus\Omega^R_{2^{-i},2^{-k}}\big).
\end{equation}
Let $x$ belonging to the right hand side of \eqref{eq:2b} and  $\eps,\eps'>0$. Pick $i\in\N$   such that $2^{-i}\leq \eps{\eps'}$ and find $j$  such that $x\in\X\setminus\Omega^R_{2^{-i},2^{-k}}$ for every $k\geq j$. Up to increase $j$ we can also assume that $2^{-j}<\frac1{R}$, then we  put $\bar r:= \eps'2^{-j}$ and  for given $r\in(0,\bar r)$ we let $k\geq j$ be such that $\eps'2^{-(k+1)}< r\leq \eps'2^{-k}$.

By definition of $\Omega^R_{2^{-i},2^{-k}}$ we know that there is $y\in \X\setminus \Omega^R_{2^{-k}}$ with $\sfd(x,y)\leq 2^{-i-k}$ and since $x\in{\rm Ann}_{R/2}(\bar x)$ the bound  $2^{-i-k}\leq 2^{-j}< \frac1{R}$ grants that $y\in {\rm Ann}_R(\bar x)\setminus \Omega^R_{2^{-k}}$. By definition of $\Omega^R_{2^{-k}}$ this means  that there exists  a unit speed geodesic starting from $\bar x$ passing through $y$ of length $\sfd(\bar x,y)+2^{-k}\geq \sfd(\bar x,y)+\frac{r}{\eps'}$ and since $\sfd(x,y)\leq 2^{-i-k}\leq \eps\eps'$, taking into account the arbitrariness of $\eps,\eps'$ we just showed  that $x\in \mathcal E_{1,\bar x}$, which was our claim.
\end{proof}
To get the proof of the dimension gap and, later, of the stratification result we shall use some facts about the $\HH^\alpha_\infty$ pre-measure. Two direct consequences of the definitions are
\begin{equation}
\label{eq:samenull}
\HH^\alpha_\infty(A)=0\qquad\Leftrightarrow\qquad\HH^\alpha(A)=0
\end{equation}
and
\begin{equation}
\label{eq:usc}
\sfd_{\rm H}(A_n,A)\to 0,\quad A\text{ compact}\qquad\Rightarrow \qquad\HH^\alpha_\infty(A)\geq \lims_{n\to\infty}\HH^\alpha_\infty(A_n).
\end{equation}
A  subtler result relates the Hausdorff measure and the density of the $\infty$-Hausdorf premeasure, see \cite[Proposition 11.3]{Giusti84} and \cite[Theorem 2.10.17]{Federer69} for the proof:
\begin{lemma}[Density of $\infty$-Hausdorff premeasure]\label{le:hinfty}
Let $(\X,\sfd)$ be a   metric space, $\alpha\geq 0$ and $E\subset \X$ a Borel set. Then for $\HH^\alpha$-a.e.\ $x\in E$ we have
\[
\lims_{r\downarrow 0}\frac{\HH_\infty^\alpha(E\cap B_r(x)) }{r^\alpha}\geq 2^{-\alpha}\omega_\alpha.
\]
\end{lemma}
A last property of Hausdorff measures that we shall use is the following, see  \cite[Theorem 2.10.45]{Federer69} for the proof:
\begin{equation}
\label{eq:dimprod}
\HH^{\alpha}(\X)=0\qquad\Leftrightarrow\qquad \HH^{\alpha+1}(\X\times\R)=0,
\end{equation}
valid for any $\alpha\geq 0$ and metric space $\X$.

We can now prove Theorem \ref{thm:dimgap}:
\begin{proofb}{\bf of Theorem \ref{thm:dimgap}} We shall assume that $\dim_\HH(\X)>N-1$ and prove that $\HH^N(\X)>0$, thanks to Proposition \ref{prop:basehn} this is sufficient to conclude. We start claiming that
\begin{equation}
\label{eq:claimtang}
\textrm{there exists an iterated tangent space of $\X$ which is the Euclidean space $\R^N$.}
\end{equation}
To prove this, let  $\eps>0$ be so that $\HH^{N-1+\eps}(\X)>0$, then by  Proposition \ref{le:nm1} we also have $\HH^{N-1+\eps}(\mathcal E_1(\X))>0$   and we can apply Lemma \ref{le:hinfty} to $E:=\mathcal E_1(\X)$ to find $x\in\mathcal E_1(\X)$ and $r_n\downarrow0$ such that 
\begin{equation}
\label{eq:hi}
\lim_{n\to\infty}\frac{\HH_\infty^{N-1+\eps}(\mathcal E_1(\X)\cap B_r(x)) }{r_n^{N-1+\eps}}\geq 2^{-\alpha}\omega_\alpha.
\end{equation} 
Recalling \eqref{eq:Gcomp}, up to pass to a not relabeled subsequence  we can assume that the spaces $(\X_n,\sfd_n,\mm_n,x_n):=(\X,\sfd/r_n,\mm/\mm(B_{r_n}(x)),x)$ pmGH-converge to some $\RCD(0,N)$ space\linebreak $(\Y^1,\sfd_{\Y^1},\mm_{\Y^1},o)$ as $n\uparrow\infty$. It is clear that after embedding all these spaces into a realization of such convergence we have that $\overline B^{\sfd_n}_1(x_n)\to \overline B^{\sfd_{\Y^1}}_1(o)$ w.r.t.\ the Hausdorff distance and thus we have
\[
\HH_\infty^{N-1+\eps}( \overline B^{\Y^1}_1(o))\stackrel{\eqref{eq:usc}}\geq\lims_{n\to\infty}\HH_\infty^{N-1+\eps}(\overline B^{\X_n}_1(x_n))=\lims_{n\to\infty}\frac{\HH_\infty^{N-1+\eps}(\overline B^{\X}_{r_n}(x))}{r_n^{N-1+\eps}}\stackrel{\eqref{eq:hi}}>0
\] 
which, by \eqref{eq:samenull}, forces 
\begin{equation}
\label{eq:hntang}
\HH^{N-1+\eps}(\Y^1)>0.
\end{equation}
By definition of $\mathcal E_1(\X)$ the fact that  $x\in \mathcal E_1(\X)$ grants that $\Y^1=\R\times \X^1$ and since  $\Y^1$ is $\RCD(0,N)$, the splitting ensures that $\X^1$ is either a point  or $N\geq 2$ and $\X^1$  is $\RCD(0,N-1)$.

If $\X^1$ is a point we have $\Y^1=\R$ and \eqref{eq:hntang} forces $N+\eps\leq 2$. Since $N\in\N$ and $N\geq 1$, this implies $N=1$.

If instead $N\geq 2$ we use \eqref{eq:hntang} and \eqref{eq:dimprod} to deduce that  $ \HH^{N-2+\eps}(\X^1)>0$ and repeat the argument with the $\RCD(0,N-1)$ space $\X^1$ in place of $\X$ and $N-1$ in place of $N$.

Iterating this procedure after exactly $N$ steps we arrive at a tangent space $\Y^N$ to the $\RCD(0,1)$ space $\X^{N-1}$ of the form $\Y^N=\R\times\X^N$, and since $\Y^N$ is itself $\RCD(0,1)$ this forces $\X^N$ to be a point.

In summary, we proved  claim \eqref{eq:claimtang}. Therefore  by a diagonalization argument there is $\tilde r_n\downarrow 0$ and $(\tilde x_n)\subset \X$ such that for the rescaled spaces $(\tilde \X_n,\tilde \sfd_n):=(\X,\sfd/\tilde r_n)$ it holds
\begin{equation}
\label{eq:itertang}
\lim_{n\to\infty}\sfd_{\rm GH}\Big(\bar B_R^{\tilde \X_n}(x_n),\bar B_R^{\R^N}(0)\Big)= 0\qquad\forall R>0.
\end{equation}
Now we consider  $\delta_2=\delta_2(1/2,N)$ be given by Proposition \ref{thm:MN2} and pick $R:=\frac1{\delta_2}$ in \eqref{eq:itertang} above to conclude that for $n$ sufficiently big we have $\tilde r_n^2K\geq-\delta_2$ and \eqref{viciniviciniGH} is satisfied for the $\RCD(\tilde r_n^2K,N)$ space $(\tilde \X_n,\tilde \sfd_n,\mm)$. Fix such $n$ and let $U\subset \tilde\X_n$, $u:U\to u(U)\subset \R^N$ be given by Proposition  \ref{thm:MN2} with $\eps=\frac12$. Notice that \eqref{eq:volest2} forces $\HH^N(u(U))=\mathcal L^N(u(U))>0$ and since $u$ is biLipschitz we also have that $U\subset \tilde \X_{n}$ has positive $\HH^N$ measure in the space $(\tilde \X_n,\tilde\sfd_n)$; given that $\tilde \X_n$ is obtained by rescaling of $\X$, we see that $U\subset \X$ also has positive $\HH^N$ measure in the space $\X$, which gives the conclusion.
\end{proofb}
{The proof of Corollary \ref{cor:ref} can now be easily obtained:
\begin{proofb}{\bf of Corollary \ref{cor:ref}} If $N$ is integer the claim is a direct consequence of Proposition \ref{prop:basehn} (see also \cite[Corollary 2.3]{Sturm06II}). Otherwise let $[N]_+:=\min\{n\in\N:N\leq n\}$ and notice that $N<[N]_+$ and that $(\X,\sfd,\mm)$ is an $\RCD(K,[N]_+)$ space. Thus by Theorem \ref{thm:dimgap} and again Proposition \ref{prop:basehn} we conclude that $\dim_\HH(\X)\leq [N]_+-1=[N]$.
\end{proofb}
}
\subsection{Non-collapsed and collapsed convergence}
Having at disposal the `continuity of volume' granted by Theorem \ref{thm:contvol} and the `dimension gap' Theorem \ref{thm:dimgap} we can now easily obtain the stability of the class of $\ncRCD(K,N)$ spaces as stated in Theorem \ref{thm:stabnc}:
\begin{proofb}{\bf of Theorem \ref{thm:stabnc}} \\
\noindent{$\mathbf{ (i)}$} The fact that the $\lims$ is actually a $\lim$ is a direct consequence of Theorem  \ref{thm:contvol}  (recall also \eqref{eq:hnsfere}).  This and the compactness of the class of $\RCD(K,N)$ spaces (see \eqref{eq:Gcomp}) ensure that up to pass to a subsequence, not relabeled, we can assume that there is a Radon measure $\mm_\infty$ on $\X_\infty$ such that $(\X_{n},\sfd_{n},\HH^N,x_{n})$ pmGH-converge to $(\X_\infty,\sfd_\infty,\mm_\infty,x)$ and to conclude it is sufficient to prove that $\mm_\infty=\HH^N$, as this will in particular imply that the limit { metric measure } space does not depend on the particular converging subsequence chosen.

By Theorems  \ref{thm:stabweak}, \ref{thm:wncpropp} and point $(i)$ of Lemma \ref{lm:propt} we  know that $\mm_\infty=\vartheta\HH^N$ for some $\vartheta\leq 1$, so that our aim is to prove that $\vartheta=1$ $\mm_\infty$-a.e.. If not, there would exist $y_\infty\in\X_\infty$ and $r>0$ such that
\begin{equation}
\label{eq:matt6}
\mm_\infty(B_r^{\X_\infty}(y_\infty))<\HH^N(B_r^{\X_\infty}(y_\infty)).
\end{equation}
Now find a  sequence $y_n\stackrel{GH}\to y_\infty$ (recall \eqref{eq:exseq}), notice that  $\sfd_{\rm GH}(\bar B_r^{\X_n}(y_n),\bar B_r^{\X_\infty}(y_\infty))\to 0$ as $n\to\infty$ and use Theorem \ref{thm:contvol} to  obtain
\[
\lim_{n\to\infty}\HH^N( B_r^{\X_n}(y_n))\stackrel{\eqref{eq:hnsfere}}=\lim_{n\to\infty}\HH^N(\bar B_r^{\X_n}(y_n))=\HH^N(\bar B_r^{\X_\infty}(y_\infty))\stackrel{\eqref{eq:matt6}}>\mm_\infty(B_r^{\X_\infty}(y_\infty)),
\]
contradicting \eqref{eq:convballs}.

\noindent{$\mathbf{ (ii)}$} We argue by contradiction and assume 
\begin{equation}
\label{eq:contrcoll}
\dim_\HH(\X_\infty)> N -1.
\end{equation}
By  the compactness of the class of $\RCD(K,N)$ spaces  we know that there exists a Radon measure $\mm_\infty$ on $\X_\infty$ and a subsequence, not relabeled, such that the normalized spaces\linebreak $(\X_n,\sfd_n,{\HH^N}/\HH^N(B_1(x_n)),x_n)$ pmGH-converge to $(\X_\infty,\sfd_\infty,\mm_\infty,x)$ (recall \eqref{eq:Gcomp}). In particular, this grants that $(\X_\infty,\sfd_\infty,\mm_\infty)$ is an $\RCD(K,N)$ space, so that our assumption \eqref{eq:contrcoll} and Theorem \ref{thm:dimgap}  yield that $\HH^N(\X_\infty)>0$ and thus there is $x'\in\X$ such that $\HH^N(B_1(x'))>0$. Now find a sequence $x'_n\stackrel{GH}\to x'$ (recall \eqref{eq:exseq}) and use   Theorem \ref{thm:contvol} (and \eqref{eq:hnsfere}) to obtain that
\[
\HH^N(B^{\X_n}_1(x_n'))\to\HH^N(B^{\X_\infty}_1(x'))>0.
\]
Taking into account that $\lim_{n\to\infty }\sfd_n(x_n,x_n')=\sfd_\infty(x,x')<\infty$, such convergence and the uniform local doubling property granted by the Bishop-Gromov inequality give that
\[
\limi_{n\to\infty}\HH^N(B^{\X_n}_1(x_n))>0,
\]
which contradicts our assumption $\lim_{n\to\infty}\HH^N(B^{\X_n}_1(x_n))=0$ and thus yields the thesis.
\end{proofb}

\subsection{Volume rigidity}
{
Collecting what proved so far it is now easy to establish the volume rigidity result, Theorem \ref{thm:volrig}, and its  Corollary \ref{cor:Bishop}.
}

\begin{proofb}{\bf of Theorem \ref{thm:volrig}}

\noindent\underline{Step 1: set up} Let  $(\X,\sfd,\HH^N)$ be a $\ncRCD(0,N)$ space and $\bar x\in\X$ such that
\begin{equation}
\label{eq:hyp}
\HH^N(B^\X_1(\bar x))\geq\HH^N(B_1^{\R^N}(0)).
\end{equation}
We shall prove that this implies that $\overline B_{1/2}^{\X}(\bar x)$ is isometric to $\overline B_{1/2}^{\R^N}(0)$. Thanks to Gromov compactness theorem \eqref{eq:Gcomp} and to the stability of the $\ncRCD$ condition under non-collapsed convergence established in Theorem \ref{thm:stabnc}, this is sufficient to conclude.

\noindent\underline{Step 2: the cone $\Y$} Consider the function $(0,1]\ni r\mapsto \frac{\HH^N(B_r(\bar x))}{\omega_N r^N}$ and notice that the Bishop-Gromov inequality grants that it is non-increasing, that by \eqref{eq:hyp} its value at $r=1$ is 1 and, recalling the definition of $\vartheta_N$ and Corollary \ref{cor:densnc}, that it converges to $\vartheta_N[\X](\bar x)\leq 1$ as $r\downarrow0$. Hence
\begin{equation}
\label{eq:dens0}
\vartheta_N[\X](\bar x)=\frac{\HH^N(B_r(\bar x))}{\omega_N r^N}=1\qquad\forall r\in(0,1]
\end{equation} 
and by the `volume cone to metric cone' \cite[Theorem 1.1]{DPG16} we get the existence of  an $N$-cone $(\Y,\sfd_\Y,\mm_\Y,o)$ and a measure preserving isometry $\iota:B^\X_{1/2}(\bar x)\to B^\Y_{1/2}(o)$.

It follows that $\mm_\Y=\HH^N$ and thus that $\Y$ is $\ncRCD(0,N)$. Also, the very definition of $\vartheta_N$ and the properties of $\iota$ give that \begin{equation}
\label{eq:xy}
\vartheta_N[\X](x)=\vartheta_N[\Y](\iota(x))\qquad\forall x\in B^\X_1(\bar x).
\end{equation} 
Now observe that by a simple scaling argument it is easy to see that $\vartheta_N[\Y]$ is constant along rays, and this fact together with the lower semicontinuity of $\vartheta_N[\Y]$ given by point $(i)$ in  Lemma \ref{lm:propt} shows that $\vartheta_N[\Y](o)\leq\vartheta_N[\Y](y)$ for every $y\in\Y$. Therefore
\begin{equation}
\label{eq:u1}
1\stackrel{\eqref{eq:dens0}}=\vartheta_N[\X](\bar x)\stackrel{\eqref{eq:xy}}=\vartheta_N[\Y](o)\leq\vartheta_N[\Y](y)\stackrel{\eqref{eq:cor}}\leq1\qquad\forall y\in\Y.
\end{equation}

\noindent\underline{Step 3: $\Y=\R^N$} According to Lemma \ref{lm:favaa} it is sufficient to prove that  $\Y$ is an $N$-metric measure cone centered at any $y\in\Y$ and by the  `volume cone to metric cone' \cite[Theorem 1.1]{DPG16} in order to prove this it is sufficient to show that $r\mapsto\frac{\HH^N(B_r(y))}{\omega_Nr^N}$ is constant for every $y\in\Y$. By the very definition of $N$-cone this is true for $y=o$, then for general $y\in\Y$ put $R:=\sfd_\Y(y,o)$ and notice that
\[
\begin{split}
\lim_{r\to\infty}\frac{\HH^N(B_r(y))}{\omega_Nr^N}= \lim_{r\to\infty}\frac{\HH^N(B_{r+R}(y))}{\omega_N(r+R)^N}\geq \lim_{r\to\infty}\frac{\HH^N(B_{r}(o))}{\omega_Nr^N}\frac{r^N}{(r+R)^N}=\vartheta_N[\Y](o)\stackrel{\eqref{eq:u1}}=1,
\end{split}
\]
so that the conclusion follows from \eqref{eq:u1} and the monotonicity granted by Bishop-Gromov inequality \eqref{e:BG}.
\end{proofb}

{
\begin{proofb}{\bf of Corollary \ref{cor:Bishop}}
Inequality \eqref{e:Bishop} immediately follows from Corollary \ref{cor:densnc} and the Bishop-Gromov inequality \eqref{e:BG}. The scaled version of Theorem \ref{thm:volrig}  ensures the desired rigidity for the equality case. Moreover, by the second step in the proof of  Theorem \eqref{thm:volrig} we see that if
\[
\vartheta(x)=\lim_{r \to 0} \frac{\HH^N(B_r^\X(x))}{\omega_Nr^N}=1,
\]
then the spaces \((\X,\sfd/r, \HH^n)\) converges to \((\R^N, \sfd_{\Eu}, \HH^N)\). The converse being an easy consequence of Theorem \ref{thm:stabnc}, this concludes the proof.
\end{proofb}
}

\subsection{Stratification}

Here we prove  the stratification result stated in Theorem \ref{thm:singular}; notice the similarity with the proof of Theorem \ref{thm:dimgap}. 

Let us begin by giving the definition of the $k$-singular set $S_k(\X)$:
\begin{equation}
\label{eq:essek}
\begin{split}
S_k(\X):=\Big\{x\in\X\ :\ &\text{for every tangent space $(\Y,\sfd_\Y,\mm_\Y,y)$ of $\X$ at $x$ we have}\\
&\text{$\sfd_{\rm GH}\big(\bar B_1^\Y(y),\bar B^{\R^{k+1}\times\Z}_1((0,z))\big)>0$ for  all pointed spaces $(\Z,\sfd_\Z,z)$}\Big\}
\end{split}
\end{equation}

We can now prove Theorem \ref{thm:singular}:
\begin{proofb}{\bf  of Theorem \ref{thm:singular}}
We argue by contradiction, thus we assume that for some  \(k\in \N\) and $\ncRCD(K,N)$ space $\X$ we have $\dim_{\HH} S_k(\X) >k$. Hence for some $k'>k$ it holds
\begin{equation}
\label{eq:kp}
\HH^{k'} (S_k(\X))>0.
\end{equation}
Then for \(\eps>0\) define
\[
S^\eps_k(\X)=\Big\{ x\in\X: \sfd_{\rm GH}\big(B_r^\X(x), B_r^{\R^{k+1}\times \Z}((0,z))\big)\ge \eps r\quad \forall r\in(0, \eps),\textrm{ pointed \(\Z\) } \Big\}
\]
and note that \(S^\eps_k(X)\) is closed and that $S_k(\X)=\bigcup_{i\in\N} S^{2^{-i}}_k(\X)$.
From this and \eqref{eq:kp} it follows that  there exists \(\bar \eps>0\) such that $\HH^{k'} (S^{\bar \eps}_k(\X))>0$.
We now apply Lemma \ref{le:hinfty} to $E:=S^{\bar \eps}_k(\X)$ to deduce that there exists $x\in S^{\bar \eps}_k(\X)$ and $r_n\downarrow0$ such that 
\begin{equation}\label{cluster}
\lim_{n\to\infty} \frac{\HH_{\infty}^{k'}\big(S^{\bar \eps}_k(\X)\cap B^{\X}_{r_n}(x)\big)}{r_n^{k'}}\ge 2^{-k'} \omega_{k'}.
\end{equation}
By Corollary \ref{cor:densnc}  we have  \(\vartheta_N[\X, \sfd, \mm](x)\leq 1\) and thus  by  Proposition \ref{le:metriccone}, up to pass to a non-relabeled subsequence, we can assume that the rescaled spaces $(\X_n,\sfd_n,\mm_n,x_n):=(\X,\sfd/r_n,\mm/r_n^N, x)$ pmGH-converge to  a pointed metric measure cone \((\Y,\sfd_\Y,\mm_\Y,y)\) which, by Theorem \ref{thm:stabnc}, is $\ncRCD(0,N)$.

Embedding all these spaces into a proper realization of this pmGH-convergence (recall \eqref{eq:Gcomp}) and using the metric version of Blaschke's theorem (see \cite[Theorem 7.3.8]{BBI01}) we see that, extracting if necessary a further subsequence, we can assume that the {compact }sets $S^{\bar \eps}_k(\X_n)\cap \bar B^{\sfd_n}_1(x_n)$ converge to some compact set $A\subset\Y$ w.r.t.\ the Hausdorff distance. A  simple diagonal argument based on the very definition of $S^{\bar \eps}_k(\Y)$ then shows that $A\subset S^{\bar \eps}_k(\Y)$ and thus
\[
\HH^{k'}_\infty (S^{\bar \eps}_k(\Y))\geq \HH^{k'}_\infty(A)\stackrel{\eqref{eq:usc}}\geq\lims_{n\to\infty}\HH_\infty^{k'}\big(S^{\bar \eps}_k(\X_n)\cap B^{\sfd_n}_1(x_n)\big)=\lims_{n\to\infty}\frac{\HH_\infty^{k'}\big(S^{\bar \eps}_k(\X)\cap B^{\sfd}_{r_n}(x) \big)}{r_n^{k'}}\stackrel{\eqref{cluster}}>0.
\]
Hence by \eqref{eq:samenull} we also have  $\HH^{k'}(S^{\bar\eps}_k(\Y)\setminus \{y\})>0$ and we can repeat the argument to find $z\in S^{\bar\eps}_k(\Y)$, $z\neq y$, and a tangent cone $(\Y',\sfd_{\Y'},\mm_{\Y'},y')$ at $z$, which is $\ncRCD(0,N)$, such that 
\begin{equation}
\label{eq:accap}
\HH^{k'}(S^{\bar \eps}_k(\Y'))>0. 
\end{equation}
Since $z\neq y$ the cone $\Y'$ contains a line passing through its origin $y'$ and thus by the splitting theorem for $\RCD$ spaces we deduce that $\Y'=\R\times \X^1$ for some  metric measure space $(\X^1,\sfd^1,\mm^1)$.

If $k=0$ this is enough to  conclude, because such splitting contradicts the choice $z\in S^{\bar\eps}_k(\Y)$. Otherwise $k\geq 1$, hence $k'>1$ and \eqref{eq:accap} and $N\in\N$ force $N\geq 2$. Then the splitting grants that $\X^1$ is an $\RCD(0,N-1)$ space and from the fact that $\Y'$ is $\ncRCD(0,N)$ and Proposition \ref{prop:prod} we  deduce that  in fact $\X^1$ is $\ncRCD(0,N-1)$. Taking into account the trivial implication
\[
(r,x^1)\in S_k(\R\times \X^1) \iff x^1\in S_{k-1}(\X^1)
\]
and \eqref{eq:dimprod}, from \eqref{eq:accap} we deduce that
\[
\dim_{\HH}(S_{k-1}(\X^1))>k-1.
\]
We can therefore repeat the whole argument with $\X^1$ and $k-1$ in place of $\X$ and $k$: iterating we eventually find a contradiction and achieve the proof.
\end{proofb}

\begin{remark}[Polar spaces]{\rm This theorem is also valid, with the same proof, in the a priori larger class of $\RCD(K,N)$ spaces $\X$ such that every iterated tangent cone is a metric cone (notice that the analogue of Proposition \ref{prop:prod} holds, rather trivially, for this class of spaces). Spaces with this property have been called \emph{polar} in \cite{Cheeger-Colding97I}.

Notice that  $\wncRCD(K,N)$ spaces such that $\vartheta_N[\X]$ is locally bounded from above are polar, and that Theorem \ref{thm:stabweak} grants that this class of spaces is stable w.r.t.\ pmGH-convergence provided we impose a uniform local upper bound on the $\vartheta$'s. 
}\fr\end{remark}
\begin{remark}[Boundary of $\ncRCD$ spaces]{\rm In the case of non-collapsed Ricci limit spaces it has been shown in \cite{Cheeger-Colding97I} that 
\begin{equation}
\label{eq:dimnm1}
S_{N-1}(\X)\setminus S_{N-2}(\X)=\emptyset.
\end{equation}
This is however false in the present situation, because, for instance, the closed unit ball $\bar B_1(0)\subset \R^N$ is a perfectly legitimate $\ncRCD(0,N)$ space and every point in the boundary belongs to $S_{N-1}(\X)\setminus S_{N-2}(\X)$.

The problem is the presence of the boundary: looking for a moment just at smooth objects, compact manifolds with (convex) boundary are always $\RCD(K,N)$ spaces for suitable $K,N$ but not considered in \cite{Cheeger-Colding97I} as objects whose limits define Ricci-limit spaces. Then in \cite{Cheeger-Colding97I} it has been proved (with an argument also linked to topology) that in the non-collapsing situation  boundary of balls converge to  boundary of balls, a fact which quite easily implies \eqref{eq:dimnm1}.

This line of thoughts suggests to define the boundary $\partial\X$ of a $\ncRCD(K,N)$ space $\X$ as
\[
\partial\X:=\text{closure of }\big(S_{N-1}(\X)\setminus S_{N-2}(\X)\big).
\] 
Then, mostly by analogy with the theory of Ricci-limit and Alexandrov spaces,  a number of natural non-trivial questions arise:
\begin{itemize}
\item[-] Given a non-collapsing sequence $\X_n\to \X$  of $\ncRCD$ spaces, is it true that $\partial\X_n$ converge to $\partial\X$?
\item[-] Is it true that either $\partial\X=\emptyset$ or  $\partial\X$ is $N-1$-rectifiable with $\HH^{N-1}\restr{\partial \X}$ locally finite?
\item[-] Is $\X\setminus\partial\X$ a convex subset of $\X$? That is, is it true that for any $x,y\in \X\setminus\partial\X$ there is a (or perhaps, is any) geodesic connecting them entirely contained in $\X\setminus\partial\X$? 
\item[-] Let $\Y$ be a connected component of  $\partial\X$. Is $\Y$ connected by Lipschitz paths? If so:
\begin{itemize}
\item let $\sfd_\Y$ be the intrinsic distance on $\Y$ induced by the distance on $\X$: is $(\Y,\sfd_\Y)$  an Alexandrov space of non-negative curvature\footnote{Some time after  the publication of this paper, we have been informed by V. Kapovitch that there is a simple counterexample to this conjecture: it is sufficient to consider as \(\X\) the unit ball centered  at the vertex of a \(N\)-cone whose section \(\Z\) is \(\RCD(N-2,N-1)\) but does not have non-negative curvature.}? (notice that the analogous of this latter question for Alexandrov spaces is open - see \cite[Section 9]{Petrunin07}).
\item Let  $\X'$ be another $\ncRCD(K,N)$ space and assume that $\Y_1,\ldots,\Y_n$ and $\Y_1',\ldots,\Y_n'$ are the connected components of $\partial\X$ and $\partial\X'$ respectively. Assume also that for any $i=1,\ldots,n$ the spaces $\Y_i$ and $\Y_i'$ with the induced length metrics are isometric and glue  $\X$ and $\X'$ along their boundaries via such isometries. Is the resulting space $\ncRCD(K,N)$? (the analogous statement for Alexandrov spaces holds, see \cite{Petrunin97}).
\end{itemize}
\end{itemize} \ 
}\fr\end{remark}

\def\cprime{$'$} \def\cprime{$'$}

  \end{document}